\documentclass{article}

\usepackage[english]{babel}
\usepackage[utf8]{inputenc}
\usepackage[T1]{fontenc}
\usepackage{ae,lmodern}
\usepackage[margin=1.0in]{geometry}
\usepackage{amsmath}
\usepackage{amssymb}
\usepackage{amsthm} 
\usepackage{float} 
\usepackage{subcaption} 
\usepackage[group-separator={,}]{siunitx} 
\usepackage{multicol}
\usepackage{makecell}
\usepackage{soul} 
\usepackage{enumitem}
\usepackage{tikz} 
\usepackage{booktabs} 
\usepackage{pgfplots} 
\pgfplotsset{compat=1.17}
\usepackage{pgfplotstable} 
\usepackage{todonotes} 
\usepackage{empheq} 
\usepackage{physics} 
\usepackage{cite} %
\usepackage[colorlinks = true,
            linkcolor = blue,
            urlcolor  = blue,
            citecolor = blue,
            anchorcolor = blue]{hyperref} 
\usepackage[capitalise,nameinlink,noabbrev]{cleveref} 
\usepackage{aliascnt}

\usepackage{fancyhdr}
\newcommand{\red}[1]{\textcolor{black}{#1}} 

\newcommand{\cyan}[1]{\textcolor{black}{#1}}

\newtheorem{definition}{Definition}
\crefname{definition}{definition}{definitions}
\Crefname{definition}{Definition}{Definitions}

\newaliascnt{theorem}{definition}
\newtheorem{theorem}[theorem]{Theorem}
\aliascntresetthe{theorem}
\crefname{theorem}{theorem}{theorems}
\Crefname{theorem}{Theorem}{Theorems}

\newaliascnt{proposition}{definition}

\aliascntresetthe{proposition}
\crefname{proposition}{proposition}{propositions}
\Crefname{proposition}{Proposition}{Propositions}

\newaliascnt{lemma}{definition}
\newtheorem{lemma}[lemma]{Lemma}
\aliascntresetthe{lemma}
\crefname{lemma}{lemma}{lemmas}
\Crefname{lemma}{Lemma}{Lemmas}

\newaliascnt{corollary}{definition}
\newtheorem{corollary}[corollary]{Corollary}
\aliascntresetthe{corollary}
\crefname{corollary}{corollary}{corollaries}
\Crefname{corollary}{Corollary}{Corollaries}

\newaliascnt{remark}{definition}
\newtheorem{remark}[remark]{Remark}
\aliascntresetthe{remark}
\crefname{remark}{remark}{remarks}
\Crefname{remark}{Remark}{Remarks}

\newaliascnt{assumption}{definition}

\aliascntresetthe{assumption}
\crefname{assumption}{assumption}{assumptions}
\Crefname{assumption}{Assumption}{Assumptions}

\newaliascnt{example}{definition}
\newtheorem{example}[example]{Example}
\aliascntresetthe{example}
\crefname{example}{example}{examples}
\Crefname{example}{Example}{Examples}

\newcommand{\todoNote}[1]%
{\todo[color=yellow!25]{#1}}

\renewcommand{\ip}[2]{\left\langle #1, #2\right\rangle }
\newcommand{\kry}[0]{\mathcal{K}}

\newcommand{\nonNegativeReals}[0]{\mathbb{R}_+}
\newcommand{\resdecrease}[0]{\nonNegativeReals^n}

\title{Any nonincreasing convergence curves are simultaneously possible for GMRES and weighted GMRES, as well as for left and right preconditioned GMRES
\footnote{This work was supported in part by the ANR project DARK (research grant ANR-24-CE46-1633).}}
\author{P. Matalon\thanks{CNRS, CMAP, École polytechnique, Institut Polytechnique de Paris}, N. Spillane\footnotemark[2]}

\date{}

\begin{document}

\maketitle

\begin{abstract}
The convergence of the GMRES linear solver is notoriously hard to predict. A particularly enlightening result by [Greenbaum, Pták, Strakoš, 1996] is that, given any \cyan{non-increasing}{} convergence curve, one can build a linear system for which \cyan{the residual norms produced by}{} GMRES realize that convergence curve. What is even more extraordinary is that the eigenvalues of the problem matrix can be chosen arbitrarily. We build upon this idea to derive novel results about weighted \cyan{and preconditioned}{} GMRES.
We prove that for any linear system and any prescribed convergence curve, there exists a weight matrix $M$ for which weighted GMRES (\emph{i.e.}\ GMRES in the inner product induced by $M$) realizes that convergence curve, and we characterize the form of $M$.
Additionally, we exhibit a necessary and sufficient condition on $M$ for the simultaneous prescription of two convergence curves, one realized by GMRES in the Euclidean inner product, and the other in the inner product induced by $M$.
These results are then applied to infer some properties of preconditioned GMRES when the preconditioner is applied either on the left or on the right. For instance, we show that any two convergence curves are simultaneously possible for left and right preconditioned GMRES.
\end{abstract}



\small\textbf{Communicated by Francoise Tisseur}

\tableofcontents

\section*{Introduction}

The GMRES algorithm \cite{saad_gmres_1986,zbMATH01953444} is one of the most popular methods for solving general non-Hermitian linear systems.
Yet, to this date, its convergence behaviour is not fully understood.
\cyan{While classical quantities such as the spectrum, the pseudo-spectrum, or the numerical range can provide insight in many situations, they do not offer a complete explanation of the observed residual decay. 
In particular, when standard sufficient conditions fail (for instance when the numerical range contains the origin), little can be concluded a priori about convergence.}

A striking change of perspective was introduced by Greenbaum, Pták and Strakoš in \cite{greenbaum_any_1996}. 
Rather than attempting to deduce convergence from properties of a fixed matrix, they considered the inverse question: given a prescribed (non-increasing) residual convergence curve, does there exist a linear system whose GMRES residuals realize it? Their fundamental result shows that any admissible convergence curve can be realized by GMRES. 
Even more remarkably, this can be achieved while also prescribing the spectrum of the matrix arbitrarily, demonstrating that eigenvalues alone do not determine GMRES convergence.

\cyan{In practice, GMRES is frequently modified either by changing the inner product (a method called weighted GMRES) or by introducing preconditioning. The purpose of this work is to analyze the convergence curves that can be achieved simultaneously either by weighted and unweighted GMRES, or by left and right preconditioned GMRES. It is shown that any two convergence curves can be simultaneously achieved as long as the exact solution is reached at the same iteration (\Cref{cor:weight-simultaneous} and \Cref{th:left_right_prec}). This indicates that modifying the inner product or the position of the preconditioner can have a drastic impact on convergence. In \Cref{th:inverse_curves}, a linear system $Ax = b$ and a preconditioner $H$ are considered. It is proved that, if right preconditioned GMRES produces residuals $\mathtt{r}_i := \| b - A x_i \|$ and left preconditioned GMRES produces residual $\widetilde{\mathtt{r}}_i := \|H(b - A x_i)\|$, there exists another linear system $\widetilde A \widetilde x = \widetilde b$ and preconditioner $\widetilde H$ such that the convergence curves are reversed. What is meant by reversed is that right preconditioning  $\widetilde A \widetilde x = \widetilde b$ by $\widetilde H$  produces the residual curve $\| \widetilde b - \widetilde A \widetilde x_i \| = \widetilde{\mathtt{r}}_i $ while left preconditioning produces the residual curve $\|\widetilde H(\widetilde b - \widetilde A \widetilde x_i ) \| = \mathtt r_i $.  } 

\cyan{The possibility for the efficiency of a preconditioner to depend on whether it is applied on the left or on the right has already been raised as early as in \cite{zbMATH03120441} (a main result of which can be found in \cite[Th 9.2]{zbMATH07687320}) and in \cite[p. 308]{zbMATH01311339}. See also the literature review in \cite{spillane2026preconditionleftright}.} A key observation is that weighted GMRES provides a unifying framework for all preconditioned variants. 
This suggests extending the prescribed convergence curve analysis to the weighted setting first, and then deduce consequences for preconditioned GMRES. 

The viewpoint in \cite{greenbaum_any_1996} has already been refined and extended in several directions. 
Arioli, Pták and Strakoš \cite{arioli_krylov_1998} provided a complete parametrization of all systems generating a prescribed convergence curve (the APS parametrization), thereby making explicit the structure of the admissible pairs $(A,b)$. 
Meurant \cite{meurant_gmres_2012} further analyzed this parametrization and clarified its implications. 
In parallel, it was shown in \cite{tebbens_any_2012} that not only the spectrum but also any prescribed set of harmonic Ritz values (\emph{i.e.}\ the zeros of the GMRES residual polynomial) can be realized for a given convergence curve at any iteration, a result further developed in \cite{duintjer_tebbens_prescribing_2014,du_any_2017}. 
Extensions to restarted and block variants of GMRES were established in \cite{vecharynski_any_2011} and \cite{zbMATH07206100}, respectively. An analysis closely related to our work can be found in \cite{MeurantError}, where the residual and error convergence curves are prescribed simultaneously. Even though the norm of the error can also be viewed as a weighted norm of the residual, this setup is particular because the weight varies with the operator.

Understanding the convergence of weighted GMRES, an interesting question in itself, can also be viewed as a step toward understanding the convergence of standard GMRES. 
In recent work \cite{embree2025extendingelmansboundgmres}, Embree has extended the range of applicability of Elman's convergence bound \cite{zbMATH03831185} by considering weighted GMRES. 
Indeed, as long as the eigenvalues are on one side of the origin, \cite{embree2025extendingelmansboundgmres} proposes the construction of an inner product for which the field of values is also on one side of the origin, and hence the Elman bound applies. We view this as a very promising way of predicting the convergence of GMRES for a large class of problems, and we aim to provide some level of insight into the connection between the convergence of GMRES in the Euclidean norm and in the norm used for analysis.

Weighted GMRES is a term coined by Essai \cite{essai_weighted_1998}, who demonstrated that the convergence of restarted GMRES can be accelerated by incorporating a diagonal weight matrix into the inner product of the Arnoldi process at each restart. 
Receiving attention, Essai's technique is analyzed from the angle of the Ritz values \cite{guttel_observations_2014} and the eigenvectors \cite{embree_weighted_2017}, providing insights into use cases where this weighted method can be particularly beneficial.
The abstract algorithm of GMRES using a non-standard inner product has been explored in \cite{pestana_nonstandard_2011,pestana_choice_2013} and references therein.
We also refer to \cite{wathen_preconditioning_2007} for a reflection about which norm should be considered to assess convergence. Earlier works on what is now called weighted GMRES include \cite{zbMATH01096035,zbMATH01271905,zbMATH01201042,zbMATH05626642,zbMATH00036024}.

\paragraph{Outline}
This manuscript starts with a preliminary section where we introduce GMRES and recall the original result from \cite{greenbaum_any_1996} about the prescription of convergence curves. The formalism for prescribed convergence curves is introduced, and in particular the notation $ g_i = \sqrt{\|r_{i-1}\|^2 - \|r_i\|^2}$ for the decrease in norm between two subsequent residuals (\cyan{as in \cite{greenbaum_any_1996}}). Prescribing a non-increasing convergence curve is the same as prescribing the quantities $g_i$ stored in a vector $g \in \mathbb R^n$. In \Cref{sec:weighted_gmres}, we present novel theoretical results about weighted GMRES and how it compares to GMRES in the Euclidean norm. Stemming from the fact that there are equivalences between weighted GMRES and preconditioned GMRES (under the right choices of linear systems, weights and preconditioners), we deduce in \Cref{sec:preconditioning} some new results about prescribed convergence curves for left and right preconditioned GMRES. To complement our theoretical results, \Cref{sec:illustrations} provides numerical illustrations exploring the behavior of weighted and preconditioned GMRES in various circumstances.
The MATLAB code reproducing the numerical experiments is openly available\footnote{\url{https://github.com/hpc-maths/2025_prescribed_gmres}}. 
Moreover, as all the proofs are constructive, it also includes the implementation of examples that verify the theorems.
Besides, we make available our MATLAB implementation of GMRES\footnote{\url{https://github.com/hpc-maths/krylov4r}}, which accepts left/right preconditioners and change of inner product.

\paragraph{Flavour of our work}

We present four results (in a simplified setting) that give a flavour of our work on weighted GMRES and on preconditioned GMRES. 

\paragraph{\cyan{Simultaneous prescription of convergence curves for weighted and unweighted GMRES.}}

The first result {(\Cref{th:there_exists_M})}{} is that, for any given system $A x = b$ and prescribed convergence curve, there exists a Hermitian positive definite weight matrix $M$ such that weighted GMRES applied to $Ax=b$ realizes the prescribed convergence curve.
The second result {(\Cref{th:two_cc_characterization})}{} considers two prescribed convergence curves (\textit{i.e.} sequences of residual norm values) \cyan{$(\mathtt{r}_i)_i$ and $(\widetilde{\mathtt{r}}_i)_i$}:
\[
 \mathtt{r}_0 > \mathtt{r}_1 > \mathtt{r}_2 >  \dots \text{ for GMRES, and }
 \widetilde{\mathtt{r}}_0 > \widetilde{\mathtt{r}}_1 > \widetilde{\mathtt{r}}_2 >  \dots  \text{ for weighted GMRES.}
\]
We prove that, for a given Hermitian positive definite weight matrix $M$, there exists a system $Ax=b$ such that both convergence curves are realized
\underline{if and only if} there exists a non-singular upper triangular matrix $T$  such that the singular values of $T$ are the inverse square roots of the eigenvalues of $M$, and
\[
\begin{pmatrix} \sqrt{\mathtt{r}_0^2 - \mathtt{r}_1^2}, & \sqrt{\mathtt{r}_1^2 - \mathtt{r}_2^2}, &  \dots\end{pmatrix}^\top = T\, \begin{pmatrix} \sqrt{{\widetilde{\mathtt{r}}_0}^2 - {\widetilde{\mathtt{r}}_1}^2},& \sqrt{{\widetilde{\mathtt{r}}_1}^2 - {\widetilde{\mathtt{r}}_2}^2},& \dots\end{pmatrix}^\top.
\]
Additionally, the eigenvalues of $A$ can be prescribed.

\paragraph{\cyan{Simultaneous prescription of convergence curves for left and right preconditioned GMRES.}}

Consider two prescribed convergence curves {$(\mathtt{r}_i)_i$ and $(\widetilde{\mathtt{r}}_i)_i$}:
\begin{itemize}
\item $\mathtt{r}_0 > \mathtt{r}_1 > \mathtt{r}_2 >  \dots $ for right preconditioned GMRES,
\item $\widetilde{\mathtt{r}}_0 > \widetilde{\mathtt{r}}_1 > \widetilde{\mathtt{r}}_2 >  \dots $  for left preconditioned GMRES.
\end{itemize}
The first result {(\Cref{th:left_right_prec})}{} is that there exists a system $A x = b$ and a preconditioner $H$ such that both convergence curves are realized. 
Additionally, the eigenvalues of $AH$ \cyan{(or $HA$, which are the same since both matrices are similar)}{} can be prescribed. 
The second result {(\Cref{th:left_right_two_cc_characterization})}{} is that, for a given preconditioner $H$, there exists a system $Ax=b$ such that both convergence curves are realized 
\underline{if and only if} there exists a non-singular upper triangular matrix $T$  such that the singular values of $T$ are the inverse of those of $H$ and
\[
\begin{pmatrix} \sqrt{\mathtt{r}_0^2 - \mathtt{r}_1^2}, & \sqrt{\mathtt{r}_1^2 - \mathtt{r}_2^2}, &  \dots\end{pmatrix}^\top = T\, \begin{pmatrix} \sqrt{\widetilde{\mathtt{r}}_0^2 - {\widetilde{\mathtt{r}}_1}^2},& \sqrt{\widetilde{\mathtt{r}}_1^2 - \widetilde{\mathtt{r}}_2^2},& \dots\end{pmatrix}^\top.
\]
Additionally, the eigenvalues of $AH$ \cyan{(or, equivalently, of $HA$)}{} can be prescribed.

\paragraph{Notation} We denote by $\nonNegativeReals$ the set of non-negative real values (including 0). Let $n\in\mathbb{N}$.
For any matrix $A\in\mathbb{C}^{n\times n}$, $A^*$ denotes its conjugate transpose and, if invertible, $A^{-1}$ its inverse and $A^{-*}$ the inverse of $A^*$.
We refer to $A$ being ``hpd'' for Hermitian positive definite.

\section{Preliminary on GMRES and convergence curve formalism} \label{sec:preliminary}

\cyan{The objective of this section is to introduce GMRES as well as a framework for convergence analysis by prescribing convergence curves that is greatly inspired by \cite{greenbaum_any_1996}.}

We denote by $(A,b)$ the general linear system 
\begin{equation*}
    Ax = b,
\end{equation*}
where $A\in\mathbb{C}^{n\times n}$ is a non-singular matrix and $b\in\mathbb{C}^n$ is a vector. 
 Given an inner product $\ip{\cdot}{\cdot}$ and its associated norm $\|\cdot\|$, we consider the application of the GMRES algorithm \cite{saad_gmres_1986,zbMATH01953444} with a zero initial guess ($x_0 = 0$) to the system $(A, b)$. We refer to this procedure as GMRES($A,b$). 
 In this section, the inner product $\ip{\cdot}{\cdot}$ remains abstract.
The method is associated with the Krylov subspaces 
\[
\kry_i(A, b) := \operatorname{span} \{b, Ab, A^2b, \dots, A^{i-1}b\} \qquad \forall i=1, \dots, n.
\]
At any iteration $i=0,\dots,\dim \kry_n(A, b)$, the approximate solution is denoted by $x_i$ and the corresponding residual by $r_i := b - A x_i$. By definition of GMRES, each residual is characterized by
    \begin{equation} \label{eq:minim_pb}
r_i \in b +  A\kry_i(A, b) \text{ and }        \|r_i\| = \min_{x\in \kry_i(A, b)} \|b-Ax\| = \min_{y\in A\kry_i(A, b)} \|b-y\|.
    \end{equation}

The Krylov residual spaces \(
    A\kry_i(A, b) = \operatorname{span}\{Ab, A^2b, \dots, A^ib\}
\) 
play a key role in formalizing GMRES. These spaces are nested by definition, and we define the notion of \emph{nested residual basis}:
\begin{definition}[Nested residual basis]
    Let $(A, b)$ define a given linear system, and $m :=\dim \kry_n(A, b)$ be the breakdown index of GMRES$(A,b)$.
    The family $(w_j)_{j=1,\dots,m}$ is called a nested basis of $A\kry_m(A, b)$ if, for all $i=1,\dots,m$, its restriction to the $i$ first vectors is a basis of $A\kry_i(A, b)$.
\end{definition}
\cyan{ The breakdown index of GMRES($A,b$) is also the iteration at which GMRES finds the exact solution \cite[Prop.\ 6.2, Prop.\ 6.10]{zbMATH01953444}. Since it does not depend on the choice of inner product, all weighted versions of GMRES($A,b$) reach the exact solution at the same iteration.}

Another interpretation of \eqref{eq:minim_pb} is that $(b-r_i)$ is the orthogonal projection of $b$ onto $A\kry_i(A, b)$. Consequently, $r_i$ can also be expressed as 
$r_i = b - \sum_{j=1}^i \ip{b}{w_j} w_j$,  
in which $(w_j)_{j=1,\dots,m}$ is a nested, orthonormal basis of $A\kry_n(A, b)$. 
 Since $r_m = 0$, this implies that $ b = \sum_{j=1}^m \ip{b}{w_j} w_j$, $r_i =  \sum_{j=i+1}^m \ip{b}{w_j} w_j$, 
    and that
\begin{equation}
\label{eq:ip_b_wj}
\abs{\ip{b}{w_j}} = \sqrt{\|r_{j-1}\|^2 - \|r_j\|^2}.
\end{equation}
\cyan{It turns out that the convergence of GMRES can be conveniently captured  by these differences (as in \cite{greenbaum_any_1996}).} 

\begin{definition}[Residual decrease vector]
\label{def:resdecvec}
Let $g \in \resdecrease$. We say that $g$ is the residual decrease vector of GMRES$(A,b)$ if  the GMRES residuals satisfy
\[
g_j = \sqrt{\|r_{j-1}\|^2 - \|r_j\|^2}, \quad j = 1,\dots, \dim \kry_n(A, b),
\]
and all remaining entries of $g$ are zero.   
\end{definition} 

Many of the new results in this article are written in the residual decrease vector formalism. Next, we make explicit the meaning of prescribing a GMRES convergence curve through a residual decrease vector. 

\begin{definition}[Length of $g$]
Let $g \in \resdecrease$. If it exists, let $p$ be the smallest index such that $g_i = 0$ for all $i>p$. Otherwise, set $p=n$. In both cases, we say that $p$ is the \emph{length} of $g$.
\end{definition}
\cyan{
It can be noticed that, if a vector $g \in \resdecrease$ is viewed as the residual decrease vector of GMRES$(A,b)$, then its length is equal to $ \dim \kry_n(A, b)$, the breakdown index of GMRES.
\begin{definition}[Realization of a convergence curve]
\label{def:PCCinduced}
Let $g \in \resdecrease$ and let $p$ be its length. 
Let $(A,b)$ define a linear system. 
We say that GMRES$(A,b)$ realizes the convergence curve induced by $g$ if, denoting by $(r_i)_i$ the residuals produced,
\[
p = \dim \kry_n(A, b) \qquad\text{ and }\qquad g_j = \sqrt{\|r_{j-1}\|^2 - \|r_j\|^2} \quad \forall j = 1,\dots, p, 
\]
or, equivalently, if
\begin{equation} \label{eq:r_from_g}
\|r_i \| = \sqrt{\sum_{j=i+1}^n g_j^2} \qquad \forall i=0,\dots,n.  
\end{equation}
\end{definition}
}

It can be observed from \eqref{eq:r_from_g} at $i=0$ that $\| b \|$ is entirely defined by $g$. \cyan{Another key observation is that prescribing residuals through \eqref{eq:r_from_g} restricts us to non-increasing convergence curves. This is a difference with \cite{MeurantError} where error curves are prescribed that are not necessarily monotonic.} For future reference, we recall in \Cref{th:existenceAb} the fundamental result of \cite[Th.\ 2.1]{greenbaum_any_1996}. We do not recall the proof which is constructive and allows to explicitly define $(A,b)$. 

\begin{theorem}[System with prescribed convergence curve and spectrum \cite{greenbaum_any_1996}]
     \label{th:existenceAb}
    Consider $g\in\resdecrease$ a residual decrease vector of arbitrary length $m$, and the prescribed, non-vanishing, complex numbers $(\lambda_i)_{i=1,\dots,n}$.\newline
    There exists a system $(A,b)$ such that the eigenvalues of $A$ are the $(\lambda_i)_{i=1,\dots,n}$ and GMRES$(A,b)$ realizes the convergence curve induced by $g$.
\end{theorem}
While the original theorem is restricted to the case where no breakdown occurs before the last iteration, early termination of GMRES is considered in \cite[Th.\ 3]{duintjer_tebbens_prescribing_2014}.  We next isolate a necessary and sufficient condition for GMRES$(A,b)$ to realize a given convergence curve. But first, we introduce two useful results about nested residual basis that are orthonormal. 

\begin{lemma}
\label{th:unicityONB}
    Let $(A, b)$ define a given linear system, and $m:=\dim \kry_n(A, b)$.
\begin{enumerate}[label=(\roman*)]
\item All nested orthonormal bases of $A\kry_m(A, b)$ are equal, up to the multiplication of each basis vector by a scalar of modulus 1.
\item There exists a nested orthonormal basis $(w_i)_{i=1,\dots,m}$  of $A\kry_m(A, b)$ such that $\ip{b}{w_i} \in \nonNegativeReals$ for all $i=1,\dots,m$.
\end{enumerate}
\end{lemma}

\begin{proof}
The first item is proved by induction over $m$. For the second, consider any nested orthonormal basis $(\widetilde{w}_i)_{i=1,\dots,m}$ of $A\kry_m(A, b)$ and set  $w_i := \gamma_i \widetilde{w}_i$ with $\gamma_i := \abs{\ip{b}{\widetilde{w}_i}}/\ip{b}{\widetilde{w}_i}$ if $\ip{b}{\widetilde{w}_i}\not=0$ and $\gamma_i :=1$ otherwise.
\end{proof}

\begin{lemma}[Necessary and sufficient condition for prescribing a convergence curve] \label{lem:characterization} 
    Let $(A, b)$ define a given linear system, and $g\in\resdecrease$ be a residual decrease vector of length $m:=\dim \kry_n(A, b)$.
    GMRES$(A,b)$ realizes the convergence curve induced by $g$ if and only if there exists  a nested, orthonormal basis of $A\kry_m(A, b)$, denoted by $(w_i)_{i=1,\dots,m}$, such that
    \begin{equation} \label{eq:b}
        b = \sum_{i=1}^m g_i w_i.
    \end{equation}
\end{lemma}
\begin{proof}First,  assume that $(w_i)_{i=1,\dots,m}$ is a nested, orthonormal basis of $A\kry_m(A, b)$ verifying \eqref{eq:b}.
   Then, by orthonormality of the $w_i$ and \eqref{eq:ip_b_wj}, the residuals $(r_i)_{i=1,\dots,m}$  of GMRES$(A,b)$ satisfy 
$        g_i = \ip{b}{w_i} =  \sqrt{\|r_{i-1}\|^2 - \|r_i\|^2}$, 
for any $ i=1, \dots, m$. 
    This shows that GMRES$(A,b)$ realizes the convergence curve induced by $g$.

Conversely, assume that GMRES$(A,b)$ realizes the convergence curve induced by $g$, \textit{i.e.},  
     $   g_i = \sqrt{\|r_{i-1}\|^2 - \|r_i\|^2}$ for any $i=1,\dots,m$. \cyan{By \Cref{th:unicityONB}, there exists  $(w_i)_{i=1,\dots,m}$, a nested orthonormal basis of $A\kry_m(A, b)$ such that $\ip{b}{w_i} \in \nonNegativeReals$ for all $i=1,\dots,m$.} Applying \eqref{eq:ip_b_wj}, we get
\begin{equation*}
    \abs{\ip{b}{w_i}} = \ip{b}{w_i} = g_i  \qquad \forall i=1, \dots, m, \qquad\text{ and }\qquad  b = \sum_{i=1}^m g_i w_i.
\end{equation*}
\end{proof}

In \Cref{lem:characterization}, it is assumed that the length $m$ of $g$ equals $\dim \kry_n(A, b)$. This is not strictly necessary because the equality could be recovered from \eqref{eq:b}. However, we prefer to keep the assumption in order to emphasize that the breakdown index is fixed by the linear system.

We end this section with the relationship between two nested residual bases. This plays a fundamental role in subsequent results since we will be considering nested bases of the same space that are orthonormal with respect to different inner products. The following result is similar to \cite[Prop.\ 1]{essai_weighted_1998} and \cite[Lem.\ 3.17]{pestana_nonstandard_2011} with the difference that we choose to consider bases for the whole of $\mathbb C^n$.

\begin{lemma}[Relationship between two nested bases]\label{lem:link_between_bases}
    Let $(A, b)$ define a given linear system, and $m := \dim \kry_n(A, b)$.
    Let $(w_i)_{i=1,\dots,n}$ be a basis of $\mathbb{C}^n$ such that $(w_i)_{i=1,\dots,m}$ is a nested basis of $A\kry_m(A, b)$. Denote by $W$ its column-matrix representation.\newline
 Consider $(\widetilde{w}_i)_{i=1,\dots,n}$ another basis of $\mathbb{C}^n$, with column-matrix $\widetilde{W}$.
    Then, $(\widetilde{w}_i)_{i=1,\dots,m}$ also forms a nested basis of $A\kry_m(A, b)$ if and only if there exists a non-singular matrix $T \in\mathbb{C}^{n\times n}$ of the form
    \begin{equation} \label{eq:T_form}
        T := \left(
            \begin{array}{c|c}
                    \widehat{T} & * \\ 
                    \hline
                    0  & *
            \end{array}
        \right),
    \end{equation}
    such that
    \begin{equation} \label{eq:W}
        \widetilde{W} = WT,
    \end{equation}
    where $\widehat{T}\in\mathbb{C}^{m\times m}$ is upper triangular and invertible, and $*$ denotes possibly non-zero coefficients of no importance to the convergence of GMRES.\newline
\end{lemma}
Note that if $m=n$ (no breakdown), then, in \eqref{eq:T_form}, $T$ is simply an upper triangular invertible matrix. 

\begin{proof}
The $(1,1)$ and $(2,1)$-blocks of $T$ are justified by the nestedness of the spaces $(A\kry_i(A, b))_{i=1,\dots,m}$. 
The remaining blocks must only ensure $T$'s non-singularity.
\end{proof}

\section{Weighted GMRES} \label{sec:weighted_gmres}

Any Hermitian positive definite (hpd) matrix $M\in\mathbb{C}^{n\times n}$ induces an inner product and a norm, respectively defined by 
\[
\ip{x}{y}_M := y^*Mx \qquad\text{ and }\qquad \|x\|_M := \sqrt{\ip{x}{x}_M} \qquad\text{ for all } x,y\in\mathbb{C}^n.
\]
\cyan{A matrix $X$ is said to have $M$-orthonormal columns when $X^*MX = I$ (the identity matrix). The notation $^*$ invariably denotes the conjugate transpose of a vector or a matrix.} 
 
We call $M$-GMRES the GMRES algorithm resulting from the choice $\ip{\cdot}{\cdot} := \ip{\cdot}{\cdot}_M$. In this formalism, $I$-GMRES refers to the Euclidean setting.
 Since the inner product $\ip{\cdot}{\cdot}$ remained abstract in the preceding section, all results presented before still hold for $M$-GMRES.
\textbf{Following \Cref{def:PCCinduced}, when stating that $M$-GMRES realizes the convergence curve induced by a residual decrease vector $g$, the residuals are measured in the $M$-norm.}

\subsection{$I$-GMRES and $M$-GMRES} \label{sec:ipsen}

Let $M$ be an hpd matrix. Let us begin with some natural relations between the $k$-th $I$-GMRES residual $r_k$ and the $k$-th $M$-GMRES residual $\widetilde r_k$. The initial guesses are still set to $0$, so the only difference between both algorithms is the norm in which the minimization property is enforced. Denoting by $\mu_{\min} (M) $ and $\mu_{\max} (M) $, the extreme eigenvalues of $M$, it follows that
\begin{equation}
\label{eq:IM-GMRES-easy}
\mu_{\min}(M)\|r_k\|_I^2 \leq \mu_{\min}(M)\| \widetilde{r}_k\|_I^2 \leq  \| \widetilde{r}_k \|_M^2  \leq \| r_k\|_M^2 \leq \mu_{\max}(M) \| r_k \|_I^2,
\end{equation}
where the first and third inequalities come from the definition of GMRES, and the others are spectral estimates. In particular, with  $\kappa(M){ = \mu_{\max}(M) / \mu_{\min} (M)}$ the condition number of $M$, we obtain
\begin{equation}\label{eq:IM-GMRES-easy2}
 \| r_k \|_I^2  \leq \| \widetilde{r}_k\|_I^2 \leq \kappa(M)  \| r_k \|_I^2 \text{ and }  \|\widetilde  r_k\|_M^2 \leq \| r_k\|_M^2 \leq  \kappa(M)  \|\widetilde  r_k\|_M^2.
\end{equation}
Additionally, $r_0 = \widetilde{r_0}$
and the normalized residuals satisfy
\begin{equation} \label{eq:link_nomalized_residuals}
\kappa(M)^{-1} \frac{\|\widetilde  r_k\|_M^2}{\|\widetilde{r}_0\|_M^2} \leq \frac{ \| r_k \|_I^2}{\|r_0\|_I^2} \leq \kappa(M)  \frac{\|\widetilde  r_k\|_M^2}{\|\widetilde{r}_0\|_M^2}.
\end{equation}

We next apply the formalism proposed by Ipsen in \cite{zbMATH01533021}[Theorem 2.1 and Corollary 2.2] with some details of the proof from \cite[Lemma 7.1]{ipsen1998different}. What is new here is that we extend (straightforwardly) the formalism to weighted GMRES and apply it to compare either the same residual in two norms, or the residuals of $I$- and $M$-GMRES. 

Consider $M$-GMRES$(A,b)$ starting with $\widetilde r_0 = b$. Up to convergence, the $k$-th residual is determined by 
\[
\widetilde r_k = b - W_k\underbrace{(W_k^* M W_k)^{-1} W_k^* M b }_{:=y_0},
\]
for any matrix $W_k \in \mathbb C^{n\times k}$ whose columns form a basis of $A \kry_k(A,b)$. 
Let $B = (b \,\,\, W_k)$ be the concatenation of the right-hand-side and this basis. This gives a basis of the Krylov subspace $\kry_{k+1}(A,b)$. Moreover, 
\[
B^* M B = \begin{pmatrix}  b^* M b & b^*M W_k \\ W_k^* M b & W_k^* M W_k \end{pmatrix},
\]
and the first row in the inverse of this matrix is \cite[eq. (0.7.3.1)]{zbMATH06125590}: 
\[
e_1^\top (B^* M B )^{-1} = (b^* M \widetilde r_k)^{-1} \begin{pmatrix}1 & -y_0^* \end{pmatrix}, \text{ where } e_1 = (1,0, \dots, 0)^\top \in \mathbb C^{k+1}.
\]
Multiplying on the right by $B^*$ gives
$e_1^\top (B^* M B )^{-1} B^* = (b^* M \widetilde r_k)^{-1} \underbrace{(b^* -y_0^*W_k^* )}_{=\widetilde r_k^*} = \frac{\widetilde r_k^*}{ \|\widetilde r_k\|_M^2 }$, 
where the last equality is because $b^* M \widetilde r_k = \|\widetilde r_k\|_M^2$ by $\widetilde r_k^* M W_k = 0$. Taking the transpose, we get 
\[
\frac{\widetilde r_k}{ \|\widetilde r_k\|_M^2 } = B (B^* M B )^{-1} e_1.
\]
Next, we take the $M$-norm and the $I$-norm of this estimate in order to get a formula for the $k$-th residual norm:
\begin{equation} \label{eq:ipsen}
 \|\widetilde r_k\|_M = \| B (B^* M B )^{-1} e_1\|_M^{-1} = \|e_1\|_{(B^* M B )^{-1}}^{-1},
\end{equation}
as well as for the ratio between the $I$- and $M$-norms of the $k$-th residual
\begin{equation*}
\frac{ \|\widetilde r_k\|_I}{ \|\widetilde r_k\|_M} = \frac{\| B (B^* M B )^{-1} e_1\|_I}{\| B (B^* M B )^{-1} e_1\|_M}.
\end{equation*}
\begin{remark}
    The identity $\|e_1\|_{(B^* M B )^{-1}} =  \|\widetilde r_k\|_M^{-1}$ can be seen directly from the fact that the $(1,1)-$block of $(B^* M B )^{-1}$ equals \cyan{$(b^* M \widetilde r_k)^{-1}= \|\widetilde r_k \|_M^{-2}$}.
\end{remark}
Of course, the same results adapt to the residuals $r_k$ of $I$-GMRES:
\[
\|r_k\|_I = \| B (B^* B )^{-1} e_1\|_I^{-1}=\|e_1\|_{(B^* B )^{-1}}^{-1} ,
\text{ and }
 \frac{\|r_k\|_M}{ \| r_k\|_I} = \frac{\| B (B^* B )^{-1} e_1\|_M }{\| B (B^* B )^{-1} e_1\|_I}.
\]
Finally, comparisons between the residual norms of unweighted and weighted GMRES can be drawn, \textit{e.g.}, 
\[
\frac{\|r_k\|_I}{\|\widetilde r_k\|_M} = \frac{\|e_1\|_{(B^*MB )^{-1}}}{\|e_1\|_{(B^*B )^{-1}}}= \frac{\| B (B^* M B )^{-1} e_1\|_M }{\| B (B^*B )^{-1} e_1\|_I },
\]
where each residual is in its natural norm,
\[
\frac{\|\widetilde r_k\|_I }{\|r_k\|_I} =   \frac{\| B (B^* M B )^{-1} e_1\|_I  \| B (B^* B )^{-1} e_1\|_I  }{\| B (B^* M B )^{-1} e_1\|_M^{2} },
\]
where both residuals are in the $I$-norm, and
\begin{equation}
\label{eq:Ipsenfinal}
\cyan{\frac{ \|\widetilde r_k\|_N}{ \| r_k\|_N} = \frac{\| B (B^* M B )^{-1} e_1\|_N}{ \| B (B^* B )^{-1} e_1\|_N }\frac{\| B (B^*B )^{-1} e_1\|_I^2}{\| B (B^* M B )^{-1} e_1\|_M^2}},
\end{equation}
where both residuals are measured in a third norm induced by some hpd matrix $N$. \cyan{A potential norm of interest could be the Euclidean norm of the error ($\|x - x_k\|_I$) which corresponds to the choice $N = (A A^*)^{-1}$.} 

We do not pursue this analysis here, but we do wish to mention two contributions of \cite{zbMATH01729203} which could also be extended to weighted GMRES: the first is a discussion on the choice of basis $W_k$, the second is the analysis of the effect on the bounds of adding a vector to $W_k$ when performing iteration $k$. 

\subsection{Prescribed convergence curves}
\label{sub:PCC}

We now turn back to the convergence curve formalism.  We start by proposing the construction of an inner product for which weighted GMRES achieves some desired convergence curve.

\begin{theorem} \label{th:there_exists_M}
    Let $(A, b)$ define a given linear system, and consider $\widetilde{g}\in\resdecrease$ a residual decrease vector.
    If $\widetilde{g}$ is of length $m:=\dim \kry_n(A, b)$, then
there exists an hpd matrix $M$ such that $M$-GMRES($A,b$) realizes the convergence curve induced by $\widetilde{g}$.
\end{theorem}
\begin{proof}
    Considering \Cref{lem:characterization}, the proof hinges on finding an hpd matrix $M$ and a basis $(\widetilde{w}_i)_{i=1,\dots,n}$ such that: $(\widetilde{w}_i)_{i=1,\dots,m}$ is an $M$-orthonormal nested basis of $A\kry_m(A, b)$, and the decomposition of $b$ in this basis writes
    \begin{equation} \label{eq:b_in_W}
        b =  \sum_{i=1}^m \widetilde{g}_i \widetilde{w}_i.
    \end{equation}
   \cyan{According to \Cref{th:unicityONB}, there exists $(w_i)_{i=1,\dots,m}$ an $I$-orthonormal nested basis of $A\kry_m(A, b)$ such that $\ip{b}{w_i} \in \mathbb R_+$ for each $i$.
 We complete the set into an ($I$-)orthonormal basis $(w_i)_{i=1,\dots,n}$ of  $\mathbb{C}^n$. 
    We denote by $W$ its column-matrix representation and let  $g := W^{-1}b = W^* b$. (Although we do not use this in the proof, we remark that $g$ is the residual decrease vector for $I$-GMRES$(A,b)$.)
}

    We define the upper triangular matrix $\widehat{T}\in\mathbb{C}^{m\times m}$ by setting all coefficients to zero, except
    \begin{equation*}
        \begin{cases}
            \widehat{T}_{ii} := g_i / \widetilde{g}_i &\text{ if } g_i \neq 0 \text{ and } \widetilde{g}_i \neq 0, \\
            \widehat{T}_{ii} := 1                     &\text{ if } g_i = \widetilde{g}_i = 0, \\
            \widehat{T}_{ii} := -g_m / \widetilde{g}_i,\quad \widehat{T}_{im} := g_m / \widetilde{g}_m &\text{ if } g_i = 0 \text{ and } \widetilde{g}_i \neq 0, \\
            \widehat{T}_{ii} := 1,\quad \widehat{T}_{im} := g_i / \widetilde{g}_m &\text{ if } g_i \neq 0 \text{ and } \widetilde{g}_i = 0.
        \end{cases}
    \end{equation*}
    Note that $\widetilde{g}_m \neq 0$ because $\widetilde{g}$ is of length $m$, and $g_m \neq 0$ because $m$ is the breakdown index of GMRES($A,b$). 
    Consequently, all diagonal coefficients are non-zero, ensuring that $\widehat{T}$ is invertible.
    We then simply define
    \begin{equation} \label{eq:T_def}
        T := \left(
            \begin{array}{c|c}
                    \widehat{T} &  \\ 
                    \hline
                      & I
            \end{array}
        \right),
    \end{equation}
    where $I$ is the identity matrix of size $(n-m)$.
    $T$ is clearly invertible.
    Consequently, by \Cref{lem:link_between_bases}, the matrix
        $\widetilde{W} := WT$
    defines a nested basis of $A\kry_m(A, b)$. 
    Additionally, one can easily verify that this definition of $T$ enforces the relation 
    \begin{equation*}
        g = T \widetilde{g},
    \end{equation*}
    which, multiplied by $W$ on the left, yields
    \begin{equation*}
        b = \widetilde{W}\widetilde{g},
    \end{equation*}
    \emph{i.e.} that relation \eqref{eq:b_in_W} holds.
    
    It remains only to find $M$ such that $\widetilde{W}$ has $M$-orthonormal columns: we set
    \begin{equation} \label{eq:M_def}
        M := (\widetilde{W} \widetilde{W}^*)^{-1},
    \end{equation}
    and verify that
    \begin{equation*}
      \widetilde{W}^* M \widetilde{W} 
      = \widetilde{W}^* (\widetilde{W} \widetilde{W}^*)^{-1} \widetilde{W} 
      = \widetilde{W}^* \widetilde{W}^{-*} \widetilde{W}^{-1}\widetilde{W}
      = I.
    \end{equation*}
    We conclude by applying \Cref{lem:characterization} in the $M$-inner product.
\end{proof}
    In the proof above, if $\widetilde{g}_i$ and $g_i$ have no zeroes, \emph{i.e.},\ if $M$-GMRES and $I$-GMRES never stagnate (or if $\widetilde{g}_i = 0 \Leftrightarrow g_i = 0$, \emph{i.e.} if $M$-GMRES and $I$-GMRES stagnate at the same iterations) then the proposed $T$ is a diagonal matrix and $\widetilde{W}$ is simply a scaling of $W$.

\begin{corollary}[Any two convergence curves are simultaneously possible] \label{cor:weight-simultaneous}
    Consider the convergence curves induced by $g$ and $\widetilde g\in\resdecrease$ of same length.\newline
    There exists a system $(A,b)$ and an hpd matrix $M$ such that $I$-GMRES$(A,b)$ realizes $g$, and $M$-GMRES$(A,b)$ realizes $\widetilde g$. Additionally, the eigenvalues of $A$ can be prescribed.
\end{corollary}
\begin{proof}
By \Cref{th:existenceAb}, there exists a system $(A,b)$ such that $I$-GMRES$(A,b)$ realizes $g$. Moreover, the eigenvalues of $A$ can be prescribed. With these $A$ and $b$, the result follows from applying \Cref{th:there_exists_M}.
\end{proof}

\cyan{
In the proof, the system $(A,b)$ is set to satisfy only the prescribed convergence curve $g$ for $I$-GMRES. Then, the second prescribed convergence curve imposes a condition on the inner product $M$ without any further restriction on $(A,b)$. For this reason, we can impose more than two convergence curves in the following sense. 
\begin{corollary}[Any number of convergence curves is simultaneously possible]
    Consider the $K + 1$ convergence curves induced by $g$ and $(\widetilde g^{(k)})_{k=1,\dots,K} \in\resdecrease$. Assume that they all have the same length.\newline
    There exists a system $(A,b)$ and $K$ hpd matrices $M^{(k)}$ such that $I$-GMRES$(A,b)$ realizes $g$, and $M^{(k)}$-GMRES$(A,b)$ realizes $\widetilde g^{(k)}$. Additionally, the eigenvalues of $A$ can be prescribed.
\end{corollary}
}
In the following theorem, we characterize the form of all weight matrices $M$ for which $M$-GMRES realizes a prescribed convergence curve.

\begin{theorem}[Characterization of the weight matrix]
\label{th:allMs}
    Let $(A,b)$ define a given linear system, {$m := \dim \kry_n(A, b)$}, and $W\in\mathbb{C}^{n\times n}$ be a unitary matrix whose first $m$ columns form a nested basis of $A\kry_m(A,b)$. Consider the residual decrease vector $\widetilde{g}\in\resdecrease$ of length $m$.
    \newline
    Consider $M$ an hpd matrix. The two following statements are equivalent:   \begin{enumerate}[label=(\roman*)]
    \item $M$-GMRES$(A,b)$ realizes the convergence curve induced by $\widetilde{g}$.
    \item There exists $T\in\mathbb{C}^{n\times n}$ of the form \eqref{eq:T_form} such that
    \begin{equation*}
        M^{-1} = W T (WT)^*
        \quad\text{ and }\quad
        b = WT\widetilde{g}.
    \end{equation*}
    \end{enumerate}
\end{theorem}
\begin{proof}
    Using \Cref{lem:characterization}, $M$-GMRES$(A,b)$ realizes the convergence curve induced by $\widetilde{g}$ if and only if there exists a matrix $\widetilde{W}$ whose columns are $M$-orthonormal, whose first $m$ columns form a nested basis of $A\kry_m(A,b)$, and such that
  $\widetilde{W} \widetilde{g} = b$. 
    By \Cref{lem:link_between_bases}, this is equivalent to: there exist a matrix $T$ of the form \eqref{eq:T_form} and a unitary matrix $W$ whose first $m$ columns form a nested basis of $A\kry_m(A,b)$ such that $\widetilde W = W T$, and $\widetilde W$ has $M$-orthonormal columns. Finally,

    \begin{equation*}
    \left\{
    \begin{aligned}
    \widetilde{W}^* M \widetilde{W} = I   \\
    \widetilde{W} \widetilde{g} =b
    \end{aligned}
    \right.
    \Leftrightarrow
    \left\{
    \begin{aligned}
     M  = (W T (WT)^*)^{-1} \\
    b = WT \widetilde{g}.
    \end{aligned}
    \right.    
    \end{equation*}
\end{proof}

\begin{remark}
    In \Cref{th:allMs} (ii), $M$ does not depend on the particular choice of $W$, since all possibilities differ only by scalar factors of modulus 1 for each vector (see \Cref{th:unicityONB}). Indeed, the multiplicative factor of modulus 1 can be carried over to the corresponding row of $T$.
\end{remark}

In the preceding results, the system $(A,b)$ was fixed, and we explored the weight matrices $M$ that enforce a prescribed convergence curve for $M$-GMRES$(A,b)$. 
We now fix the weight matrix $M$, and investigate the existence of a system $(A,b)$ such that both the convergence curves of $I$-GMRES$(A,b)$ and $M$-GMRES$(A,b)$ are prescribed.

\begin{theorem}[Simultaneous prescription of two convergence curves] \label{th:two_cc_characterization}
    Let $M\in\mathbb{C}^{n\times n}$ be an hpd matrix, and consider the residual decrease vectors $g$ and $\widetilde{g}\in\resdecrease$ of matching length. Let $\lambda_1$,\dots, $\lambda_n$ be $n$ non-zero complex numbers. 
    The two following statements are equivalent:
\begin{enumerate}[label=(\roman*)]
    \item\label{it:two-cc-1} There exists a system $(A,b)$ such that the eigenvalues of $A$ are the $(\lambda_i)_i$, $I$-GMRES($A,b$) realizes $g$ and $M$-GMRES($A,b$) realizes $\widetilde{g}$.
    \item\label{it:two-cc-2} There exists a non-singular matrix $T$ of the form \eqref{eq:T_form} such that $g = T\widetilde{g}$ and the singular values of $T$ are $(1/\sqrt{\mu_i})_i$ with $(\mu_i)_i$ denoting the (positive) eigenvalues of $M$.
\end{enumerate}
\end{theorem}
Remark that, once more, there is no condition at all on the spectrum of $A$. In \Cref{sub:analysisth14}, the conditions on $T$ in the second item of the theorem are analyzed and discussed.
\begin{proof}
    Assume that \ref{it:two-cc-1} holds.
    According to \Cref{lem:characterization}, there exists an $I$-orthonormal (resp.\ $M$-orthonormal) basis of $\mathbb{C}^n$ such that $(w_i)_{i=1,\dots,m}$ (resp.\ $(\widetilde{w}_i)_{i=1,\dots,m}$) is a nested basis of $A\kry_m(A, b)$ verifying
    \begin{equation} \label{eq:b_Wg_Wt_gt}
        b = Wg = \widetilde{W}\widetilde{g}.
    \end{equation}
    According to \Cref{lem:link_between_bases} \ref{it:two-cc-1}, there exists a non-singular matrix $T$ of the form \eqref{eq:T_form} such that
    \begin{equation} \label{eq:Wt_WT}
        \widetilde{W} = WT.
    \end{equation}
    Plugging \eqref{eq:Wt_WT} into \eqref{eq:b_Wg_Wt_gt}, we infer $g = T\widetilde{g}$. Additionally, $\widetilde{W}$ having $M$-orthonormal columns writes
    \begin{equation*}
        I = \widetilde{W}^*M\widetilde{W} = T^*W^*MWT \quad \Leftrightarrow \quad 
 (TT^*)^{-1}=W^*M W  \quad \Leftrightarrow \quad  TT^* = W^* M^{-1} W.
    \end{equation*}
    Since \cyan{$W^* = W^{-1}$}, the matrices $TT^*$ and $M^{-1}$ are similar, so they share the same set of eigenvalues $(1/\mu_i)_i$. Consequently, the singular values of $T$ are the $(1/\sqrt{\mu_i})_i$.
    
    Conversely, assume that \ref{it:two-cc-2} holds. Then, $TT^*$ and $M^{-1}$ are two hpd matrices that share the same set of eigenvalues, so there exists a unitary matrix $W$ such that 
    \begin{equation}
        \label{eq:similar}
        TT^* = W^* M^{-1} W.
    \end{equation}
    Let us pose $b := Wg$. 
    By \Cref{th:existenceAb}, there exists a matrix $A$ such that the eigenvalues of $A$ are the $(\lambda_i)_i$, and $I$-GMRES($A,b$) realizes the convergence curve induced by $g$. 
    We next check that $M$-GMRES($A,b$) realizes the convergence curve induced by $\widetilde{g}$.  
    Let $m := \dim \kry_n(A, b)$.
    Letting $\widetilde{W} := WT$ ensures that the first $m$ columns of $\widetilde{W}$ form a nested basis of $A\kry_m(A, b)$ and that $b = \widetilde{W}\widetilde{g}$. By \Cref{lem:characterization}, it remains only to prove that $\widetilde{W}$ is $M$-orthonormal, which comes from 
    \[
 \widetilde{W}^*M\widetilde{W} =  T^*W^*MWT = I
    \]
    by \eqref{eq:similar}.
   
\end{proof}

\begin{theorem}
\label{eq:bound_svd_T}
Let $(A,b)$ define a given linear system and $M$ (hpd) define an inner product. Let $g$ and $\widetilde{g}$ be the residual decrease vectors realized by $I$- and $M$-GMRES, respectively. Let $m$ be their common length. By \Cref{th:two_cc_characterization}, there exists a non-singular matrix $T$ of the form \eqref{eq:T_form} such that $g = T\widetilde{g}$. We denote by $T_{k+1:m}$ the block of $T$ consisting of the rows and columns between indices $k+1$ and $m$. If, $\sigma_{\min}(T_{k+1:m})$ and  $\sigma_{\max}(T_{k+1:m})$ denote the extreme singular values of this matrix, then 
\begin{equation} \label{eq:r_sigma}
\sigma_{\min}(T_{k+1:m})   \|\widetilde r_k \|_M \leq \|r_k\|_I \leq \sigma_{\max}(T_{k+1:m}) \|\widetilde r_k \|_M,
\end{equation}
where, as usual, $r_k$ and $\widetilde{r}_k$ are the $k$-th residual vectors produced by $I$- and $M$-GMRES($A,b$), respectively. 
\end{theorem}
    \begin{proof}   
 In this proof we use the (MATLAB inspired) notation $y(a:b)$ (respectively, $Y(a:b,a:b)$) to select elements in a vector $y$ (respectively, submatrices in a matrix $Y$) that correspond to indices between $a$ and $b$.
Since $g = T \widetilde{g}$ with $T$ of the form \eqref{eq:T_form}, and only the first $m$ entries of $g$ and $\widetilde{g}$ are possibly non-zero,
\[
 g(1:m) = T(1:m,1:m) \widetilde{g}(1:m) .
\]
Moreover, $\widehat{T} = T(1:m,1:m)  $ is upper triangular, so for any $k\leq m-1$,
\[
 g(k+1:m) =  T(k+1:m,k+1:m)  \widetilde{g}(k+1:m) =  T_{k+1:m}  \widetilde{g}(k+1:m),
\]
with the notation introduced in the theorem. It follows that 
\[
\|g(k+1:m)\|_I^2 =  \|\widetilde{g}(k+1:m)\|^2_{ T_{k+1:m}^*  T_{k+1:m} },
\]
and
\begin{equation*}
 \sigma_{\min}( T_{k+1:m})^2 \|\widetilde{g}(k+1:m)\|_I^2 \leq \| g(k+1:m)\|_I^2 \leq \sigma_{\max}^2( T_{k+1:m}) \|\widetilde{g}(k+1:m)\|_I^2,
\end{equation*}
which is equivalent to the result in the theorem by definition of the residual decrease vectors. 
    \end{proof}

By an interlacing theorem for singular values \cite[Corollary 7.3.6]{zbMATH06125590}, 
\[
[\sigma_{\min}( T_{k+1:m}), \sigma_{\max}( T_{k+1:m})] \subset [\sigma_{\min}( T), \sigma_{\max}( T)].
\]
Together with \Cref{eq:bound_svd_T} and \eqref{eq:r_from_g}, this provides a new proof for \eqref{eq:IM-GMRES-easy}. 

\begin{remark}[Comparison with {\cite[Theorem 3.19]{pestana_nonstandard_2011}}]
The previous result reminds us of a result by Pestana \cite[Theorem 3.19]{pestana_nonstandard_2011}. There, two nested bases for the Krylov subspace $\kry_k(A,b)$ at iteration $k$ are denoted by $V_k^I$ and $V_k^M$, where the first is $I$-orthonormal and the second is $M$-orthonormal. Letting $R_k$ be the non-singular upper triangular matrix, such that $V_k^I = V_k^M R_k$, the author proves that
\begin{equation} \label{eq:pestana_r_sigma}
\sigma_{\min}(R_{k+1}^{-1})   \|\widetilde r_i \|_M \leq \|r_i\|_I \leq \sigma_{\max}(R_{k+1}^{-1}) \|\widetilde r_i \|_M .
\end{equation}
A significant difference between both bounds is that $R_m$ links two Krylov bases while our $\widehat{T}$ links two Krylov residual bases. 
A consequence is that the quantities in the bound in \cite{pestana_nonstandard_2011} can be computed at the considered iteration whereas our bound looks ahead.  
\end{remark}

\subsection{Analysis of the necessary and sufficient condition in \Cref{th:two_cc_characterization}}
\label{sub:analysisth14}
It appears to us that \Cref{th:two_cc_characterization} is crucial in predicting whether the convergence of GMRES$(A,b)$ can be significantly modified by changing the inner product. The necessary and sufficient condition in item~\ref{it:two-cc-2} of \Cref{th:two_cc_characterization} is unfortunately not very natural. Assuming for simplicity that there is no breakdown, there are $n(n+1)/2$ coefficients in $T$. \cyan{The condition $g = T \widetilde g$ produces one linear constraint within each row of $T$ which leaves $n(n-1)/2$ degrees of freedom to get the correct singular values.} 

For a $2 \times 2$ matrix, we can examine the compatibility between ensuring $g = T \widetilde g$ and setting the singular values to $\sigma_1$ and $\sigma_2$.  Let $T = \begin{pmatrix} a & b \\ 0 & d\end{pmatrix}$, $g=(g_1,g_2)$, $\widetilde g = (\widetilde g_1,\widetilde g_2)$, with $g_2,\widetilde g_2 \neq 0$. If $g = T \widetilde g$ then 
\begin{subequations} \label{eq:abd}
\begin{equation}
d = \frac{g_2}{\widetilde g_2} \text{ and } b  = \frac{g_1 - a \widetilde g_1 }{\widetilde g_2}, 
\end{equation}
so only $a$ is still free. However, the product of the singular values equals $\det(T)$ so,
\begin{equation}
\sigma_1 \sigma_2 = ad = \frac{ a g_2}{\widetilde g_2} \, \Leftrightarrow a = \frac{\sigma_1 \sigma_2 \widetilde g_2}{g_2},
\end{equation}
\end{subequations}
and $T$ has been entirely set. In order for $T$ to satisfy the conditions in item~\ref{it:two-cc-2} of \Cref{th:two_cc_characterization}, it remains to ensure one final condition which can take the form of a condition on $\sigma_1^2 + \sigma_2^2$ 
\begin{equation} \label{eq:condition_singular_values}
\sigma_1^2 + \sigma_2^2 = \tr(T^* T) = a^2 + b^2 + d^2 = \left(\frac{\sigma_1 \sigma_2 \widetilde g_2}{g_2}\right)^2 + \left( \frac{g_1 g_2 - {\sigma_1 \sigma_2 \widetilde g_2} \widetilde g_1}{g_2 \widetilde g_2} \right)^2  + \left(  \frac{g_2}{\widetilde g_2}  \right)^2.
\end{equation}
Unless $\sigma_1$, $\sigma_2$, $g$ and $\widetilde g$ satisfy the above, there is no system $(A,b)$ such that $I$-GMRES realizes $g$ and $M$-GMRES realizes $\widetilde g$. 

\cyan{
\begin{example}
Consider the weight matrix $M := \begin{pmatrix} 1 & 0 \\ 0 & 1/4\end{pmatrix}$ and the residual decrease vectors $g := \begin{pmatrix} 2 \\ 1\end{pmatrix}, \widetilde{g} := \begin{pmatrix} 1 \\ 1\end{pmatrix}$.
\Cref{th:two_cc_characterization} \ref{it:two-cc-2} indicates that $T$ must have the singular values $\sigma_1 = 1, \sigma_2 = 2$, which, through formulas \eqref{eq:abd}, yields $T = \begin{pmatrix} 2 & 0 \\ 0 & 1\end{pmatrix}$.
One can easily check that condition \eqref{eq:condition_singular_values} is verified.
As a consequence, \Cref{th:two_cc_characterization} \ref{it:two-cc-1} holds (for any desired spectrum for the matrix $A$).\newline
On the other hand, set $g_1 := 3$ while keeping the rest unchanged. Formulas \eqref{eq:abd} fix $T = \begin{pmatrix} 2 & 1 \\ 0 & 1\end{pmatrix}$, and one can check that \eqref{eq:condition_singular_values} is not verified, which means that the singular values $\sigma_1 = 1, \sigma_2 = 2$ are not enforced on $T$.
Consequently, there exists no system $(A, b)$ such as described in \Cref{th:two_cc_characterization} \ref{it:two-cc-1}.
\end{example}
}

For $n > 2$, the situation is not as clear, and we list below some requirements.

\begin{lemma}[Necessary conditions for \ref{it:two-cc-2} in \red{\Cref{th:two_cc_characterization}}] 
    Let $\sigma_1 \geq , \dots \geq \sigma_n$ be $n$ positive real numbers. Consider two residual decrease vectors $g, \widetilde{g} \in \mathbb R_+^n$ that share the same length $m$. If there exists $T$ of the form \eqref{eq:T_form} such that $g = T \widetilde{g}$ and the singular values of $T$ are the $(\sigma_i)_i$, then
    \begin{enumerate}[label=(\roman*)]
    \item $T_{mm} = \widehat T_{mm} = {g}_m / \widetilde{g}_m$ (recall that $\widehat T$ is the $(1,1)$-block of $T$).

    \item If there is no breakdown, $T$ is triangular and the diagonal values of $T$ are the eigenvalues of $T$.
    \item The first $m$ diagonal values of $T$ are eigenvalues of $T$.
\item The eigenvalues of $T$, denoted by $\xi_1, \dots, \xi_n$ and ordered so that $|\xi_1]\geq |\xi_2| \geq \dots \geq |\xi_n|$ satisfy the Weyl conditions \cite[Problem 7.3.P17]{zbMATH06125590}
\[
|\xi_1 \dots \xi_k| \leq \sigma_1 \dots \sigma_k, \text{for each } k =1,\dots n,
\]
and the inequality is an equality for $k=n$
\[
|\xi_1 \dots \xi_n| = \sigma_1 \dots \sigma_n.
\]
Two special cases are $k=1$ and $k=n$ for which we get
$| \xi_1| \leq \sigma_1 \text{ and } |\xi_{n}| = \frac{|\xi_1 \dots \xi_n|}{|\xi_1 \dots \xi_{n-1}|} \geq   \sigma_n.$
There is also an additive variant of the Weyl conditions
\[
|\xi_1| + \dots + |\xi_k| \leq \sigma_1 + \dots  + \sigma_k, \text{for each } k =1,\dots n.
\]
    \item Matrix $T$ has $\widetilde{g}_m / g_m$ as one of its eigenvalues so
    \[
    \sigma_n \leq \widetilde{g}_m / g_m \leq \sigma_1.
    \]
\item The \cyan{squared}{} Frobenius norm of $T$ satisfies
\[
\sum_{i=1}^n \sum_{j=1}^n |T_{ij}|^2 = \sum_{i=1}^n \sigma_i^2.
\]
\end{enumerate}
\end{lemma}
\begin{proof}
\begin{enumerate}[label=(\roman*)]
   \item In the condition $g = T\widetilde g$, the blocks denoted by $*$ in \eqref{eq:T_form} play no role because the corresponding entries of $g$ and $\widetilde{g}$ are zero. Moreover, the $(1,1)$ block $\widehat T$ of $T$ is upper triangular and $\widetilde{g}_m \neq 0$. 
   \item Straightforward.
   \item The first $m$ diagonal values of $T$ are the diagonal values of $\widehat T$. These are also the eigenvalues of $\widehat T$. Moreover, the eigenvalues of $\widehat T$ are also eigenvalues of $T$ since:
   \[
   \widehat T y = \lambda y \Rightarrow T \begin{pmatrix} y \\ 0 \end{pmatrix} = \lambda  \begin{pmatrix} y \\ 0 \end{pmatrix},
   \]
   where the $0$ block in the eigenvector of $T$ is in $\mathbb R^{n-m}$.
   \item From the literature. 
\item This is a combination of the first, third and fourth items.  
\item By definition.
    \end{enumerate}
\end{proof}

\section{Preconditioned GMRES} \label{sec:preconditioning}

While weighted GMRES may sometimes be considered as a theoretical tool, rather than a practical method to solve real-world problems, preconditioned GMRES is widely employed to enhance the convergence of the standard $I$-GMRES method.
Yet, both are linked, and preconditioned GMRES can be interpreted as weighted GMRES applied to a modified system.
In this section, we explore what our findings on weighted GMRES imply on preconditioned GMRES.

A preconditioner can be applied \emph{on the left} or \emph{on the right}. If both are used, we speak of \emph{split preconditioning}.
We consider this general case directly by denoting a pair of left and right preconditioners as $(H_L,H_R)$. One of them can be set to $I$ to fall back on one-sided preconditioning.
$I$-GMRES preconditioned on the left by $H_L$ and on the right by $H_R$ is defined as $I$-GMRES applied to the linear system
\begin{equation*}
    H_L A H_R u = H_Lb, \qquad x = H_Ru.
\end{equation*}
At the $i$-th iteration, $I$-GMRES produces the iterate $u_i$ (from which the approximate solution $x_i := H_Ru_i$ can be deduced) and the residual $\widetilde{r}_i := b - A H_R u_i$.
Denoting by $H := H_R H_L$ the combined preconditioner, $\widetilde{r}_i$ is characterized by
\begin{equation*}
\|H_L \widetilde{r}_i\|_I = \min_{u\in \kry_i(H_L A H_R, H_L b)} \|H_L b- H_L A H_R u\|_I 
                      = \min_{u\in H_L\kry_i(AH, b)} \|b-AH_Ru\|_{H_L^* H_L}
                      = \min_{x\in H\kry_i(AH, b)} \|b-Ax\|_{H_L^* H_L}.
\end{equation*}
Finally, introducing the variable $v := H^{-1}x$ and the matrix $\widehat{A} := AH$, the characterization rewrites
\begin{equation} \label{eq:min_split}
\|H_L \widetilde{r}_i\|_I = \min_{v\in \kry_i(AH, b)} \|b-AHv\|_{H_L^* H_L}
                      = \min_{v\in \kry_i(\widehat{A}, b)} \|b-\widehat{A}v\|_{H_L^* H_L}.
\end{equation}
This formulation as well as \eqref{eq:minim_pb} show that the $i$-th residual $\widetilde{r}_i$ of $I$-GMRES$(A,b)$ split-preconditioned by $(H_L,H_R)$ is equal to the $i$-th residual of non-preconditioned $(H_L^* H_L)$-GMRES$(\widehat{A},b)$. We say that $I$-GMRES$(A,b)$ split-preconditioned by $(H_L,H_R)$ is \textit{equivalent} to non-preconditioned $(H_L^* H_L)$-GMRES$(\widehat{A},b)$.  
Specializing formulation \eqref{eq:min_split} for special choices of $(H_L,H_R)$ allows us to derive the following equivalent methods:

\begin{enumerate}[label=(\Alph*)]
\item \label{it:spliteq} If $M$ is an hpd matrix that admits the factorization $M = P^* P$, then $M$-GMRES$(A,b)$ is equivalent to $I$-GMRES$(A,b)$ split preconditioned by $(P, P^{-1})$.
\item \label{it:right_prec_eq} $I$-GMRES($A,b$) right preconditioned by $H$ is equivalent to unpreconditioned $I$-GMRES($\widehat{A},b$).
\item \label{it:left_prec_eq} $I$-GMRES($A,b$) left preconditioned by $H$ is equivalent to unpreconditioned $(H^*H)$-GMRES($\widehat{A},b$).
\end{enumerate}
We can also make two fundamental remarks:
\begin{itemize}
\item The choice of preconditioning on the left or on the right, or of split preconditioning does not modify the Krylov space in which the approximate solution $x_i$ is optimized. The minimization space depends only on the combined preconditioner $H = H_R H_L$. 
\item The norm in which the residual $r_i = b - A x_i$ is minimized depends only on the left preconditioner through $H_L^* H_L$. The fact that preconditioning on the right (\textit{i.e.}, setting $H_L = I$) does not modify the minimized norm is a rather well known fact.
\end{itemize} 

\cyan{
\subsection{Left, right and split preconditioned GMRES}
Based on the equivalences above, the results from \Cref{sec:ipsen} can be rewritten for left, right and split preconditioned GMRES. We give below just two such examples that correspond to \eqref{eq:IM-GMRES-easy2} and \eqref{eq:Ipsenfinal} (see also \cite[Sec. 3.1]{spillane2026preconditionleftright}). 
\\
Let the residuals of $I$-GMRES$(A,b)$ right preconditioned by $H$ be denoted by $r_k$ and the 
residuals of $I$-GMRES$(A,b)$ left preconditioned by $H$ be denoted by $\widetilde r_k$ (which means that the residual norm $\|H \widetilde r_k\|_I$ is actually produced). The left/right preconditioned version of \eqref{eq:IM-GMRES-easy} is  
\begin{equation}\label{eq:prec-GMRES-easy2}
 \| r_k \|_I  \leq \| \widetilde{r}_k\|_I \leq \kappa(H)  \| r_k \|_I \text{ and }  \|H\widetilde r_k\|_I \leq \|H r_k\|_I \leq  \kappa(H)  \|H\widetilde  r_k\|_I, 
\end{equation}
where $\kappa(H) := \|H\|_I\|H^{-1}\|_I = \sqrt{\kappa(H^* H)}$ is the condition number of $H$. 
\\
The more involved bound is 
\begin{equation}
\label{eq:precIpsenfinal}
\frac{ \|\widetilde r_k\|_N}{ \| r_k\|_N} = \frac{\| B (B^* H^* H B )^{-1} e_1\|_N}{ \| B (B^* B )^{-1} e_1\|_N }\frac{\| B (B^*B )^{-1} e_1\|_I^2}{\|H B (B^* H^* H B )^{-1} e_1\|_I^2},
\end{equation}
where $N$ is any hpd matrix,  $B = (b \,\,\, W_k)$ is the concatenation of the right-hand-side with any matrix $W_k \in \mathbb C^{n\times k}$ whose columns form a basis of $AH \kry_k(AH,b)$, and $e_1 = (1,0, \dots, 0)^\top \in \mathbb C^{k+1}$.
}

\subsection{Prescribed convergence curves}

We now turn to extending the results of \Cref{sub:PCC} to preconditioned GMRES. Following \eqref{eq:min_split} and the subsequent remarks, we naturally extend the concept of realizing a convergence curve to preconditioned GMRES. 

\cyan{
\begin{definition}[Realization of a convergence curve for preconditioned GMRES]
\label{def:PCCinducediprec}
Let $g \in \resdecrease$ and let $p$ be its length. Let $(A,b)$ define a linear system. 
Consider $I$-GMRES$(A,b)$ split preconditioned by $(H_L,H_R)$, producing $x_i$ as the $i$-th approximate solution.
Denoting $r_i := b-Ax_i$, we say that the preconditioned method realizes the convergence curve induced by $g$ if
\[
p = \dim \kry_n(A, b) \qquad\text{ and }\qquad g_j = \sqrt{\|H_L r_{j-1}\|_I^2 - \|H_L r_j\|_I^2} \qquad \forall j = 1,\dots, p, 
\]
or, equivalently, if
\begin{equation} 
\|H_L r_i \|_I = \sqrt{\sum_{j=i+1}^n g_j^2} \qquad \forall i=0,\dots,n.  
\end{equation}
\end{definition}
}

As a direct consequence of item \ref{it:spliteq} above, \Cref{th:there_exists_M} implies that

\begin{theorem}[Split preconditioning] \label{th:split_precond}
    Let $(A, b)$ define a given linear system, and consider $\widetilde{g}\in\resdecrease$ a residual decrease vector.
    If $\widetilde{g}$ is of length $m:=\dim \kry_n(A, b)$, then
    there exists a split preconditioner such that $I$-GMRES($A,b$) realizes the convergence curve induced by $\widetilde{g}$.
\end{theorem}
\begin{proof}
{    By \Cref{th:there_exists_M}, there exists $M$ hpd such that $M$-GMRES realizes the convergence curve. Let $M = P^* P$ be the Cholesky factorization of $M$. The prescribed residual norms are the $\|r_i\|_M = \| P r_i\|_I$. By item \ref{it:spliteq} above we conclude that $I$-GMRES preconditioned by $(P,P^{-1})$ also realizes the convergence curve (in the sense of \Cref{def:PCCinducediprec}).} 
\end{proof}

In the following result, we show that it is possible to simultaneously prescribe the convergence curves for the GMRES algorithm preconditioned on the left and on the right. 

\begin{theorem}[Any two convergence curves are simultaneously possible] \label{th:left_right_prec}
    Consider the convergence curves induced by $g_L$ and $g_R\in\resdecrease$ of same length.\newline
    There exists a system $(A,b)$ and a preconditioner $H$ such that $I$-GMRES$(A,b)$ right preconditioned by $H$ realizes $g_R$, and $I$-GMRES$(A,b)$ left preconditioned by $H$ realizes $g_L$. Additionally, the eigenvalues of $AH$ \cyan{(or, equivalently, of $HA$)}{} can be prescribed.
\end{theorem}
\begin{proof}
 By \Cref{cor:weight-simultaneous}, there exists a system $(\widehat A,b)$ and an hpd matrix $M$ such that $I$-GMRES$(\widehat A,b)$ realizes $g_R$, and $M$-GMRES$(\widehat A,b)$ realizes $g_L$. Additionally, the eigenvalues of $\widehat A$ can be prescribed.
 Since $M$ is hpd, there exists $H$ such that
       $ M = H^*H$.
    We moreover define
    $    A := \widehat{A}H^{-1}$.

    Referring to \ref{it:right_prec_eq}, $I$-GMRES($\widehat{A},b$) is equivalent to $I$-GMRES($A,b$) right preconditioned by $H$, which therefore realizes $g_R$. Similarly, referring to \ref{it:left_prec_eq}, $M$-GMRES($\widehat{A},b$) is equivalent to $I$-GMRES($A,b$) left preconditioned by $H$, which therefore realizes $g_L$. Additionally, the eigenvalues of $AH$ (which are also the eigenvalues of $HA$) can be prescribed.
\end{proof}

The following result translates \Cref{th:allMs} into the preconditioning formalism.

\begin{theorem}[Characterization of the left preconditioner]
\label{th:allHs}
    Consider $(A,b)$ and $H$ a given linear system and a preconditioner. Let $\widehat A := AH$, $m := \dim \kry_n(\widehat A, b)$, and $W\in\mathbb{C}^{n\times n}$ be a unitary matrix whose first $m$ columns form a nested basis of $\widehat A\kry_m(\widehat A,b)$. Consider the residual decrease vector ${g_L}\in\resdecrease$ of length $m$.
 The two following statements are equivalent:  
    \begin{enumerate}[label=(\roman*)]
    \item $I$-GMRES$(A,b)$ preconditioned on the left by $H$ realizes the convergence curve induced by $g_L$.
    \item There exists $T\in\mathbb{C}^{n\times n}$ of the form \eqref{eq:T_form} such that
    \begin{equation*}
        (H^* H)^{-1} = W T (WT)^*
        \quad\text{ and }\quad
        b = WTg_L.
    \end{equation*}
    \end{enumerate}
\end{theorem}
\begin{proof}
By applying \Cref{th:allMs}, to the linear system $(\widehat A, b)$ and the weight $H^* H$, the two following statements are equivalent:   
\begin{enumerate}[label=(\roman*)]
    \item $(H^*H) $-GMRES$(\widehat A,b)$ realizes the convergence curve induced by $g_L$.
    \item There exists $T\in\mathbb{C}^{n\times n}$ of the form \eqref{eq:T_form} such that
    \begin{equation*}
        (H^*H)^{-1} = W T (WT)^*
        \quad\text{ and }\quad
        b = WT g_L.
    \end{equation*}
    \end{enumerate}
The proof is completed by recalling item~\ref{it:left_prec_eq} in the list of equivalences.  
\end{proof}

The following result is a corollary of \Cref{th:left_right_prec}: for any given system and preconditioner, regardless of the convergence curves achieved by left and right preconditioned GMRES, there exists another preconditioned system with the same spectrum that exhibits the opposite behavior.
This implies that, when considering all possible pairs of systems and preconditioners, neither left nor right preconditioning is inherently superior to the other.
We give a constructive proof using \Cref{th:allHs}.

{
\begin{corollary} \label{th:inverse_curves}
    Let $(A,b)$ define a given linear system, and $H$ a preconditioner.
    Denote by $g_L$ the residual decrease vector realized by $I$-GMRES$(A,b)$ preconditioned by $H$ on the left, and 
by $g_R$ the residual decrease vector realized by $I$-GMRES$(A,b)$ preconditioned by $H$ on the right.
\newline
    There exists a system $(\widetilde{A},\widetilde{b})$ and a preconditioner $\widetilde{H}$ such that: 
\begin{itemize}
\item $I$-GMRES$(\widetilde{A},\widetilde{b})$ realizes $g_R$ when preconditioned by $\widetilde{H}$ on the left, and $g_L$ when preconditioned by $\widetilde{H}$ on the right; 
\item  $\widetilde{A}\widetilde{H}$ has the same spectrum as $AH$.
\end{itemize}
\end{corollary}
}
\begin{proof}
Let $\widetilde{H}$ be any non-singular matrix, and set $\widetilde{A} := HA\widetilde{H}^{-1}$ so that $\widetilde{A}\widetilde{H} = HA$. $I$-GMRES$(\widetilde{A},Hb)$ right preconditioned by $\widetilde{H}$ realizes $g_L$.  
Indeed, $I$-GMRES$(A,b)$ preconditioned by $H$ on the left realizes $g_L$ and is defined as unpreconditioned $I$-GMRES$(HA,Hb)$ which, in turn, is equivalent to $I$-GMRES$(HA\widetilde{H}^{-1},Hb)$ preconditioned on the right by $\widetilde{H}$ (through \ref{it:right_prec_eq}).
\\
It remains to choose $\widetilde{H}$ so that  $I$-GMRES$(\widetilde{A},Hb)$ left preconditioned by $\widetilde{H}$ realizes $g_R$. 
Unpreconditioned $I$-GMRES$(HA,Hb)$ realized $g_L$ so, by \Cref{lem:characterization}, there exists a unitary matrix $W \in \mathbb C^{n\times n}$ such that the first columns in $W$ form a nested basis of $HA\kry_n(HA,Hb)$ and
$H b = W g_L$.
\\
Moreover, there exists $T\in\mathbb{C}^{n\times n}$ of the form \eqref{eq:T_form} such that $g_L = T g_R$ (by the same construction as in the proof of \Cref{th:there_exists_M} knowing that $g_R$ and $g_L$ are of the same length by definition). Now, we have
$H b = W T g_R$.
If we finally set $ \widetilde{H} := (WT)^{-1}$, then $({\widetilde H}^* \widetilde{H})^{-1} = W T (WT)^*$ and we conclude by \Cref{th:allHs} that GMRES$(\widetilde A, H b)$ preconditioned on the left by  $ \widetilde{H}$ realized $g_R$. 
\end{proof}

We end with the preconditioned GMRES version of \Cref{th:two_cc_characterization}.

\begin{theorem}[Simultaneous prescription of two convergence curves] \label{th:left_right_two_cc_characterization}
    Let $H\in\mathbb{C}^{n\times n}$ be a non-singular matrix, and consider the residual decrease vectors $g_L$ and $g_R\in\resdecrease$. We assume that $g_R$ and $g_L$ have the same length. Let $\lambda_1$,\dots, $\lambda_n$ be $n$ non-zero complex numbers. 
    The two following statements are equivalent:
    \begin{enumerate}[label=(\roman*)]
    \item \label{it:one} There exists a system $(A,b)$ such that the eigenvalues of $A H$ are the $(\lambda_i)_i$, $I$-GMRES($A,b$) preconditioned on the right by $H$ realizes $g_R$, and $I$-GMRES($A,b$) preconditioned on the left by $H$ realizes $g_L$.
    \item \label{it:two} There exists a non-singular matrix $T$ of the form \eqref{eq:T_form} such that $g_R = T g_L$ and the singular values of $T$ are $(1/\mu_i)_i$ with $(\mu_i)_i$ denoting the singular values of $H$.
    \end{enumerate}
\end{theorem}
The eigenvalues $\lambda_i$ do not play a role in the second assertion, meaning that they can be independently prescribed.
\begin{proof}
Statement \ref{it:one} is equivalent to 
    \begin{enumerate}[label=\textit{(iii)}]
    \item \label{it:three} There exists a system $(A,b)$ such that the eigenvalues of $A H$ are the $(\lambda_i)_i$, $I$-GMRES($AH,b$) realizes $g_R$ and $(H^* H)$-GMRES($A H,b$) realizes $g_L$.
    \end{enumerate}
We next prove that assertions \textit{\ref{it:two}} and \ref{it:three} are equivalent.
    Let $M = H^* H$. The eigenvalues of $M$ are the squared singular values of $H$, which we denote by $(\mu_i^2)_i$. Next, we apply \Cref{th:two_cc_characterization}, which states that item \textit{\ref{it:two}} is equivalent to: 
     there exists a system $(\widehat A,b)$ such that the eigenvalues of $\widehat A$ are the $(\lambda_i)_i$, $I$-GMRES($\widehat A,b$) realizes $g_R$ and $M$-GMRES($\widehat A,b$) realizes $g_L$. It simply remains to set $A := \widehat A H^{-1}$ to recover \ref{it:three}.
\end{proof}

The necessary and sufficient condition in \Cref{th:left_right_two_cc_characterization} has been analyzed in \Cref{sub:analysisth14}. 

From these results, we can make the following observation: there exist preconditioned systems where the choice of left or right preconditioning leads to significantly different GMRES convergence behaviours. For instance, left preconditioning may result in fast convergence, while right preconditioning may lead to slow convergence, even if in both cases, the eigenvalues of the preconditioned system are perfectly clustered. 
This observation reinforces what is already known: GMRES convergence is not solely determined by the eigenvalues. 
It also highlights that in certain cases, the choice of preconditioning on the left or on the right may not be trivial and lead to significant differences.

\section{Illustrations} \label{sec:illustrations}

The purpose of the following experiments is to gain insight on the links between the convergence curves of $I$- and $M$-GMRES and the weight matrix $M$. 
 The experiments are run in MATLAB, and the scripts are openly available\footnote{\url{https://github.com/hpc-maths/2025_prescribed_gmres}}.
 Previous numerical comparisons can be found in \cite{embree_weighted_2017,essai_weighted_1998,guttel_observations_2014} for restarted GMRES, and \cite{pestana_choice_2013,embree2025extendingelmansboundgmres} for particular choices of the weight that arise from analyzing GMRES convergence.

A word of warning about the following results is that we only plot normalized residuals{, so that all convergence curves start from the same initial point}. This is the same as scaling $b$ so that $\|b\|_I = 1$ and scaling $M$ so that $\|b\|_M = 1$ also. For this reason, only the \cyan{relative distribution} of eigenvalues of $M$ is significant. 
Moreover, we consider for simplicity that all prescribed convergence curves are of length $n$ (\emph{i.e.} no early breakdown), so that all matrices $T$ of the form \eqref{eq:T_form} are triangular.

\subsection{Fixed spectrum for $M$ and convergence curve for $I$-GMRES, variable $T$}


We place ourselves in the formalism of \Cref{th:two_cc_characterization}. 
Our focus is on examining how the convergence of $M$-GMRES is affected by prescribing the spectrum of $M$. 
Recall that for a given $I$-GMRES residual decrease vector $g\in\resdecrease$ and a given spectrum $\mu\in\nonNegativeReals^n$ for $M$, the possible convergence curves for $M$-GMRES correspond to the set of all $\widetilde{g} := T^{-1}g$, such that the singular values of $T$ are the $(1/\sqrt{\mu_i})_i$, which means that the singular values of $T^{-1}$ are the $(\sqrt{\mu_i})_i$.

The experimental setup of this section is then the following: given a prescribed convergence curve for $I$-GMRES and a prescribed spectrum $\mu$ for $M$, we
\begin{enumerate}
    \item construct the residual decrease vector $g$ by \Cref{def:resdecvec},
    \item randomly generate a matrix $T^{-1}$ with singular values $(\sqrt{\mu_i})_i$ (see below),
    \item compute $\widetilde{g} := T^{-1}g$,
    \item recover the residuals norms realized by $M$-GMRES through formula \eqref{eq:r_from_g}.
\end{enumerate}

In practice, to randomly generate a triangular matrix $T^{-1}$ with prescribed singular values, we apply the following procedure:
(i) generate a random orthogonal matrix $V$ by performing the QR factorization of a general matrix with randomly generated coefficients; (ii) define $T^{-1}$ as the triangular matrix in the QR factorization of $\Sigma V$, where $\Sigma$ denotes the diagonal matrix holding the prescribed singular values. 

Browsing through all possible $T^{-1}$ matrices spans all possible convergence curves for $M$-GMRES.
However, in practice, randomizing the coefficients of $T^{-1}$ by this procedure does not produce the complete variety of possible convergence behaviors.
Indeed, first, the initial random matrix is drawn from a uniform distribution, which inherently restricts the diversity of resulting matrices (and thus, the range of orthogonal matrices $V$ that can be obtained). Second, the QR factorization of $\Sigma V$ produces a fully populated triangular matrix, excluding structurally sparse forms (such as diagonal matrices).
Despite these limitations, we examine the convergence curves generated by this method as representative examples within the broader landscape of $M$-GMRES behavior.

In the following experiments, we consider real matrices and we set $n=20$.

\paragraph{Experiment 1 (one eigenvalue jump)}
The spectrum of $M$ is chosen to contain two distinct eigenvalues, 1 and $10^{12}$, with respective multiplicities $(n-k)$ and $k$, with $k=8$.
\Cref{fig:experiment1} displays convergence curves of $M$-GMRES obtained from a sample of 5 random matrices $T^{-1}$.
Each subfigure corresponds to a prescribed convergence curve for $I$-GMRES: (a) complete stagnation until breakdown; (b) linear decay of the residual in logarithmic scale; (c) general, irregular decrease of the residual.

We observe that all samples initially follow the convergence pattern of $I$-GMRES. 
Subsequently, they exhibit a sudden and significant jump before resuming the residual decrease characteristic of $I$-GMRES.
According to formula~\eqref{eq:link_nomalized_residuals}, the magnitude of this jump is bounded by $\sqrt{\kappa(M)}$, here $10^6$. 
The plots indicate that this upper bound is closely approached.
Notably, the jump occurs at iteration 8, which aligns with the order of multiplicity $k$ of the large eigenvalue. 
Experiments with varying values of $k$ confirm that the jump consistently occurs at iteration $k$. This can be understood by recalling that $T^{-1}$ is the triangular matrix in the QR factorization of $\Sigma V$, a matrix that has $8$ rows that are of much larger magnitude than the others. 

Additionally, when the spectrum of $M$ consists not of two distinct eigenvalues but of two clusters of eigenvalues randomly distributed around the same two values, the results remain qualitatively unchanged.
Note also that permutations of the vector $\mu$ does not change the results either.

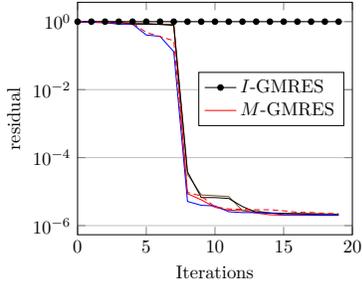
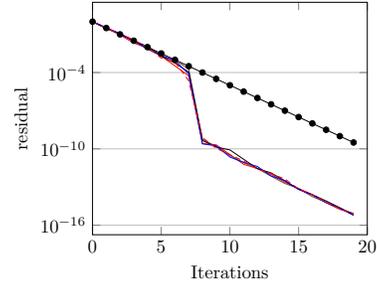
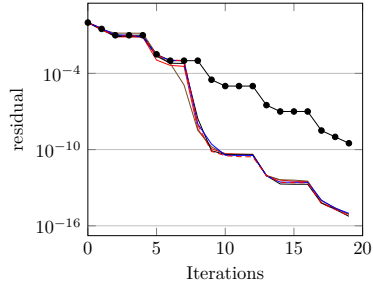
\begin{figure}
  \begin{center}
  \begin{subfigure}[b]{0.49\textwidth} 
    \begin{tikzpicture}[scale=0.7]
      \begin{semilogyaxis}[
          width=6.8cm,
          height=6cm,
          ymajorgrids,
          xlabel={Iterations},
          ylabel={residual},
          xmin=0, xmax=20,
          legend cell align=left,
          legend style={at={(0.97,0.7)}, anchor=north east}
        ]
        \addplot [color=black, mark=*, mark size=1.5pt] table [x expr=\coordindex, y index=0]{data/rW__stagnate__two_distinct.txt};
        \addlegendentry{$I$-GMRES}
        \foreach \i in {1,...,5} {
            \addplot+ [no markers] table [x expr=\coordindex, y index=\i] {data/rW__stagnate__two_distinct.txt};
        }
        \addlegendentry{$M$-GMRES}
      \end{semilogyaxis}
      \end{tikzpicture}
      \caption{The $I$-GMRES convergence curve is set to complete stagnation.}
        \label{fig:experiment1:stagnation}
    \end{subfigure}
    \hfill
    \begin{subfigure}[b]{0.49\textwidth} 
        \centering
        \begin{tikzpicture}[scale=0.7]
      \begin{semilogyaxis}[
          width=6.8cm,
          height=6cm,
          ymajorgrids,
          xlabel={Iterations},
          ylabel={residual},
          xmin=0, xmax=20,
          legend cell align=left,
          legend pos=outer north east,
        ]
        \addplot [color=black, mark=*, mark size=1.5pt] table [x expr=\coordindex, y index=0] {data/rW__linear_decay__two_distinct.txt};
        \foreach \i in {1,...,5} {
            \addplot+ [no markers] table [x expr=\coordindex, y index=\i] {data/rW__linear_decay__two_distinct.txt};
        }
      \end{semilogyaxis}
      \end{tikzpicture}
      \caption{The $I$-GMRES convergence curve is set to linear decay in log scale.}
        \label{fig:experiment1:linear}
    \end{subfigure}
    \vskip\baselineskip
    \begin{subfigure}[b]{0.49\textwidth} 
        \centering
        \begin{tikzpicture}[scale=0.7]
      \begin{semilogyaxis}[
          width=6.8cm,
          height=6cm,
          ymajorgrids,
          xlabel={Iterations},
          ylabel={residual},
          xmin=0, xmax=20,
          legend cell align=left,
          legend pos=outer north east,
        ]
        \addplot [color=black, mark=*, mark size=1.5pt] table [x expr=\coordindex, y index=0] {data/rW__irregular__two_distinct.txt};
        \foreach \i in {1,...,5} {
            \addplot+ [no markers] table [x expr=\coordindex, y index=\i] {data/rW__irregular__two_distinct.txt};
        }
      \end{semilogyaxis}
      \end{tikzpicture}
      \caption{The $I$-GMRES convergence curve is set to a general irregular decrease.}
        \label{fig:experiment1:irregular}
    \end{subfigure}
    \caption{Experiment 1. Convergence curves of $M$-GMRES for a sample of random $T^{-1}$ matrices. $M$ has two distinct eigenvalues with large gap: 1 and $10^{12}$, with respective multiplicities 12 and 8. \cyan{In the plots, the convergence curve of $I$-GMRES is represented in black with dot markers, while all other lines (colored, without markers) represent those of various instances of $M$-GMRES.}}
    \label{fig:experiment1}
  \end{center}
\end{figure}

\paragraph{Experiment 2 (two eigenvalue jumps)}

The spectrum of $M$ is chosen to contain three distinct eigenvalues: 1, $10^{10}$ and $10^{20}$, with respective multiplicities $(n-k_1-k_2)$, $k_1$ and $k_2$, with $k_1:=7, k_2:=4$.
\Cref{fig:experiment2} illustrates the results under the same experimental conditions as Experiment 1.
The plots are similar, except that two jumps are observed due to the two jumps in the eigenvalues.
In this scenario, the combined magnitude of both jumps is theoretically constrained by $\sqrt{\kappa(M)} = 10^{10}$, which is also closely approached in this case.
Concerning the location of the jumps: the first jump occurs at iteration $k_2=4$, which corresponds to the order of multiplicity of the largest eigenvalue; the second jump occurs $k_1$ iterations later, which corresponds to the order of multiplicity of the second-largest eigenvalue, here $11$.

\begin{figure}
  \begin{center}
  \begin{subfigure}[b]{0.49\textwidth} 
    \begin{tikzpicture}[scale=0.7]
      \begin{semilogyaxis}[
          width=6.8cm,
          height=6cm,
          ymajorgrids,
          xlabel={Iterations},
          ylabel={residual},
          xmin=0, xmax=20,
          legend cell align=left,
          legend style={at={(0.97,0.8)}, anchor=north east}
        ]
        \addplot [color=black, mark=*, mark size=1.5pt] table [x expr=\coordindex, y index=0]{data/rW__stagnate__three_distinct.txt};
        \addlegendentry{$I$-GMRES}
        \foreach \i in {1,...,5} {
            \addplot+ [no markers] table [x expr=\coordindex, y index=\i] {data/rW__stagnate__three_distinct.txt};
        }
        \addlegendentry{$M$-GMRES}
      \end{semilogyaxis}
      \end{tikzpicture}
      \caption{$I$-GMRES: complete stagnation.}
        \label{fig:experiment2:stagnation}
    \end{subfigure}
    \hfill
    \begin{subfigure}[b]{0.49\textwidth} 
        \centering
        \begin{tikzpicture}[scale=0.7]
      \begin{semilogyaxis}[
          width=6.8cm,
          height=6cm,
          ymajorgrids,
          xlabel={Iterations},
          ylabel={residual},
          xmin=0, xmax=20,
          legend cell align=left,
          legend pos=outer north east,
        ]
        \addplot [color=black, mark=*, mark size=1.5pt] table [x expr=\coordindex, y index=0] {data/rW__linear_decay__three_distinct.txt};
        \foreach \i in {1,...,5} {
            \addplot+ [no markers] table [x expr=\coordindex, y index=\i] {data/rW__linear_decay__three_distinct.txt};
        }
      \end{semilogyaxis}
      \end{tikzpicture}
      \caption{$I$-GMRES: linear decay in log scale}
        \label{fig:experiment2:linear}
    \end{subfigure}
    \vskip\baselineskip
    \begin{subfigure}[b]{0.49\textwidth} 
        \centering
        \begin{tikzpicture}[scale=0.7]
      \begin{semilogyaxis}[
          width=6.8cm,
          height=6cm,
          ymajorgrids,
          xlabel={Iterations},
          ylabel={residual},
          xmin=0, xmax=20,
          legend cell align=left,
          legend pos=outer north east,
        ]
        \addplot [color=black, mark=*, mark size=1.5pt] table [x expr=\coordindex, y index=0] {data/rW__irregular__three_distinct.txt};
        \foreach \i in {1,...,5} {
            \addplot+ [no markers] table [x expr=\coordindex, y index=\i] {data/rW__irregular__three_distinct.txt};
        }
      \end{semilogyaxis}
      \end{tikzpicture}
      \caption{$I$-GMRES: general irregular decrease}
        \label{fig:experiment2:irregular}
    \end{subfigure}
    \caption{Experiment 2. Convergence curves of $M$-GMRES for a sample of random $T^{-1}$ matrices. $M$ has three distinct eigenvalues with large gap: 1, $10^{10}$, $10^{20}$, with respective multiplicities 9, 7 and 4. \cyan{In the plots, the convergence curve of $I$-GMRES is represented in black with dot markers, while all other lines (colored, without markers) represent those of various instances of $M$-GMRES.}}
    \label{fig:experiment2}
  \end{center}
\end{figure}
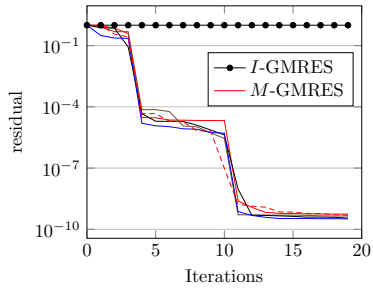
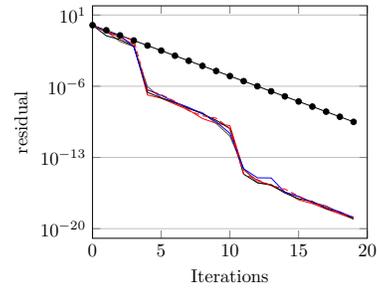
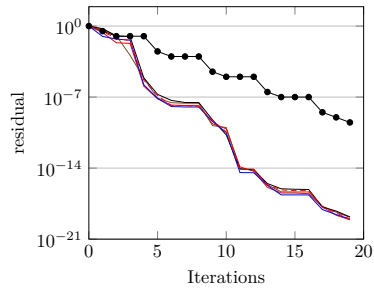

\paragraph{Experiment 3 (uniform distribution)}

The spectrum of $M$ is uniformly distributed on a logarithmic scale between 1 and $10^{12}$.
\Cref{fig:experiment3} illustrates the results under the same experimental conditions as the previous experiments.
Based on our earlier interpretations, these curves adhere to the convergence pattern of $I$-GMRES, supplemented by small decreases at each iteration. 
Given that the eigenvalues are equally spaced on a logarithmic scale, the resulting jumps are of similar magnitude. 
Furthermore, since all eigenvalues are distinct, these small jumps manifest at each iteration.

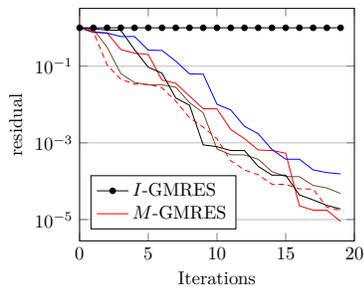
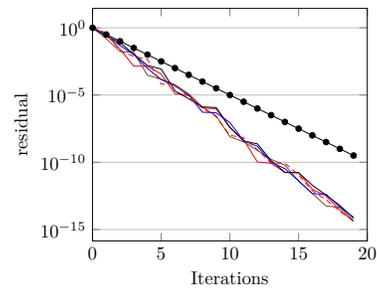
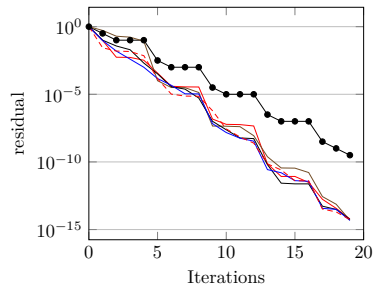
\begin{figure}
  \begin{center}
  \begin{subfigure}[b]{0.49\textwidth} 
    \begin{tikzpicture}[scale=0.7]
      \begin{semilogyaxis}[
          width=6.8cm,
          height=6cm,
          ymajorgrids,
          xlabel={Iterations},
          ylabel={residual},
          xmin=0, xmax=20,
          legend cell align=left,
          legend style={at={(0.04,0.05)}, anchor=south west}
        ]
        \addplot [color=black, mark=*, mark size=1.5pt] table [x expr=\coordindex, y index=0]{data/rW__stagnate__evenly_spaced.txt};
        \addlegendentry{$I$-GMRES}
        \foreach \i in {1,...,5} {
            \addplot+ [no markers] table [x expr=\coordindex, y index=\i] {data/rW__stagnate__evenly_spaced.txt};
        }
        \addlegendentry{$M$-GMRES}
      \end{semilogyaxis}
      \end{tikzpicture}
      \caption{$I$-GMRES: complete stagnation}
        \label{fig:experiment3:stagnation}
    \end{subfigure}
    \hfill
    \begin{subfigure}[b]{0.49\textwidth} 
        \centering
        \begin{tikzpicture}[scale=0.7]
      \begin{semilogyaxis}[
          width=6.8cm,
          height=6cm,
          ymajorgrids,
          xlabel={Iterations},
          ylabel={residual},
          xmin=0, xmax=20,
          legend cell align=left,
          legend pos=outer north east,
        ]
        \addplot [color=black, mark=*, mark size=1.5pt] table [x expr=\coordindex, y index=0] {data/rW__linear_decay__evenly_spaced.txt};
        \foreach \i in {1,...,5} {
            \addplot+ [no markers] table [x expr=\coordindex, y index=\i] {data/rW__linear_decay__evenly_spaced.txt};
        }
      \end{semilogyaxis}
      \end{tikzpicture}
      \caption{$I$-GMRES: linear decay in log scale}
        \label{fig:experiment3:linear}
    \end{subfigure}
    \vskip\baselineskip
    \begin{subfigure}[b]{0.49\textwidth} 
        \centering
        \begin{tikzpicture}[scale=0.7]
      \begin{semilogyaxis}[
          width=6.8cm,
          height=6cm,
          ymajorgrids,
          xlabel={Iterations},
          ylabel={residual},
          xmin=0, xmax=20,
          legend cell align=left,
          legend pos=outer north east,
        ]
        \addplot [color=black, mark=*, mark size=1.5pt] table [x expr=\coordindex, y index=0] {data/rW__irregular__evenly_spaced.txt};
        \foreach \i in {1,...,5} {
            \addplot+ [no markers] table [x expr=\coordindex, y index=\i] {data/rW__irregular__evenly_spaced.txt};
        }
      \end{semilogyaxis}
      \end{tikzpicture}
      \caption{$I$-GMRES: general irregular decrease}
        \label{fig:experiment3:irregular}
    \end{subfigure}
    \caption{Experiment 3. Convergence curves of $M$-GMRES for a sample of random $T^{-1}$ matrices. The spectrum of $M$ is uniformly distributed on a logarithmic scale between 1 and $10^{12}$. \cyan{In the plots, the convergence curve of $I$-GMRES is represented in black with dot markers, while all other lines (colored, without markers) represent those of various instances of $M$-GMRES.}}
    \label{fig:experiment3}
  \end{center}
\end{figure}

\paragraph{Experiment 4 (diagonal $T^{-1}$)}

While the previous experiments explored fully populated triangular matrices $T^{-1}$, we now consider diagonal matrices.
The experimental setup is the following: the convergence curve of $I$-GMRES is set to a linear decay in logarithmic scale, and $T^{-1}$ is the diagonal matrix $\operatorname{diag}(\sqrt{\mu})$, \emph{i.e.}\ with coefficients $(\sqrt{\mu_i})_i$.
Several settings for $\mu$ are tried out, all based on two eigenvalues, 1 and $10^{12}$, in various orders:
\begin{align*}
    \underline{\mu}^k &:= [\underbrace{1, \dots, 1}_{k}, \underbrace{10^{12}, \dots, 10^{12}}_{n/2}, \underbrace{1, \dots, 1}_{n/2-k}]^* &\forall k=0,\dots,n/2, \\
    \overline{\mu}^k &:= [\underbrace{10^{12}, \dots, 10^{12}}_{k}, \underbrace{1, \dots, 1}_{n/2}, \underbrace{10^{12}, \dots, 10^{12}}_{n/2-k}]^* &\forall k=0,\dots,n/2.
\end{align*}
The results are displayed in \Cref{fig:experiment4}. 
We observe that $M$-GMRES with $\mu := \underline{\mu}^k$ stagnates during $k$ iterations, follows the residual decrease of $I$-GMRES during $n/2$ iterations, then jumps and \cyan{follows} $I$-GMRES for the remaining iterations.
On the other hand, with $\mu := \overline{\mu}^k$: we notice that $\overline{\mu}^0 = \underline{\mu}^{n/2}$ and $\overline{\mu}^{n/2} = \underline{\mu}^0$ (in red); for all the other values of $k$, the convergence of $M$-GMRES follows that of $I$-GMRES during $(k-1)$ iterations, then jumps at the $k$-th iteration, and finally stagnates.

\begin{figure}
  \begin{center}
  \begin{subfigure}[b]{0.49\textwidth} 
    \begin{tikzpicture}[scale=0.8]
      \begin{semilogyaxis}[
          width=6.8cm,
          height=6cm,
          ymajorgrids,
          xlabel={Iterations},
          ylabel={residual},
          xmin=0, xmax=20,
          legend cell align=left,
          legend style={at={(0.04,0.05)}, anchor=south west}
        ]
        \addplot [color=black, mark=*, mark size=1.5pt] table [x expr=\coordindex, y index=0]{data/experiment4_1.txt};
        \addlegendentry{$I$-GMRES}
        \foreach \i in {1,...,11} {
            \addplot+ [no markers] table [x expr=\coordindex, y index=\i] {data/experiment4_1.txt};
        }
        \addlegendentry{$M$-GMRES}
      \end{semilogyaxis}
      \end{tikzpicture}
      \caption{$\mu := \underline{\mu}^k$ for $k=0,\dots,n/2$}
        \label{fig:experiment4:1}
    \end{subfigure}
    \hfill
    \begin{subfigure}[b]{0.49\textwidth} 
        \centering
        \begin{tikzpicture}[scale=0.8]
      \begin{semilogyaxis}[
          width=6.8cm,
          height=6cm,
          ymajorgrids,
          xlabel={Iterations},
          ylabel={residual},
          xmin=0, xmax=20,
          legend cell align=left,
          legend pos=outer north east,
        ]
        \addplot [color=black, mark=*, mark size=1.5pt] table [x expr=\coordindex, y index=0]{data/experiment4_2.txt};
        \foreach \i in {1,...,11} {
            \addplot+ [no markers] table [x expr=\coordindex, y index=\i] {data/experiment4_2.txt};
        }
      \end{semilogyaxis}
      \end{tikzpicture}
      \caption{$\mu := \overline{\mu}^k$ for $k=0,\dots,n/2$}
        \label{fig:experiment4:2}
    \end{subfigure}
    \caption{Experiment 4. \cyan{In the plots, the convergence curve of $I$-GMRES is represented in black with dot markers, while all other lines (colored, without markers) represent those of various instances of $M$-GMRES.}}
    \label{fig:experiment4}
  \end{center}
\end{figure}

\subsection{Can we break the observed behaviour?}


In the first three experiments, all generated samples exhibited the same convergence pattern.
In the following experiments, we demonstrate that this behavior can, in fact, be disrupted by setting up an experiment for which we can compute the residuals of weighted and unweighted GMRES as follows.
\paragraph{Experiment 5}
\begin{enumerate}
    \item Set $\mu \in(\nonNegativeReals\setminus\{0\})^n$. Let $M_0 := Q \operatorname{diag}(\mu) Q^*$, where $\operatorname{diag}(\mu)$ is the diagonal matrix corresponding to $\mu$, and $Q$ is a random unitary matrix. This way, $M_0$ is an hpd matrix with eigenvalues held in $\mu$. 
    \item Let $b := \frac{1}{\sqrt{n}}Q [1, \dots, 1]^\top$ be the right-hand side.
    \item Normalize $M_0$ as $M := \beta M_0$, where $\beta := \|b\|_{M_0}^{-2}$ (so as to enforce $\|b\|_M = 1$). This is our weight matrix.
    \item  We consider a matrix $A$ for which $W := Q$ is a nested orthonormal basis of $A\kry_n(A,b)$ and compare $I$-GMRES$(A,b)$ to $M$-GMRES$(A,b)$. The residuals do not depend on the particular choice of $A$ (by \eqref{eq:ip_b_wj}). There are two ways of computing the resulting residual norms:
 \begin{itemize}
    \item Either, the matrix $A$ can be constructed following the procedure in the proof of \cite[Th.\ 3]{duintjer_tebbens_prescribing_2014} ({a result that extends \cite[Th.\ 2.1]{greenbaum_any_1996} recalled in \Cref{th:existenceAb}}). Then $I$-GMRES$(A,b)$ and $M$-GMRES$(A,b)$ are run and the corresponding residual norms collected.
\item Or, the residuals can be computed directly. By \eqref{eq:ip_b_wj}, denoting by $w_i$ the $i$-th column of $W$, $\sqrt{\|r_{i-1}\|_I^2 - \|r_i\|_I^2} = |\langle b,w_i \rangle_I | = \frac{1}{\sqrt{n}}$ and $\sqrt{\|r_{i-1}\|_M^2 - \|r_i\|_M^2} = \abs{\langle b, \sqrt{\frac{{\mu}_i}{\beta}}w_i \rangle_M} = \sqrt{\frac{\mu_i}{n \beta}}$ because the scaled vectors $(\sqrt{\frac{{\mu}_i}{\beta}} w_i)_i$ form an $M$-orthonormal nested basis of  $A\kry_n(A,b)$.
    \end{itemize}
\end{enumerate}

The convergence curves are shown in \Cref{fig:experiment5}. With $n = 60, \epsilon = 10^{-2}$ and $p=n/2$, three cases are considered for $\mu$:
\[
    \mu^1 := [\underbrace{1, \dots, 1}_{p}, \underbrace{\epsilon, \dots, \epsilon}_{n-p}]^\top, \text{ for which }
    \|\widetilde{r}_i\|_M = 
    \begin{cases}
        \sqrt{1 - \frac{\beta}{n}i} &\text{if } i \leq p, \\
        \sqrt{\|\widetilde{r}_{p}\|_M^2 - \frac{\beta \epsilon}{n}(i-p)} & \text{if } i > p,
    \end{cases}
\]
\[ 
    \mu^2 := [\underbrace{\epsilon, \dots, \epsilon}_{p}, \underbrace{1, \dots, 1}_{n-p}]^\top, \text{ for which }
    \|\widetilde{r}_i\|_M = 
    \begin{cases}
        \sqrt{1 - \frac{\beta\epsilon}{n}i} \quad \approx 1 &\text{if } i \leq p, \\
        \sqrt{\|\widetilde{r}_{p}\|_M^2 - \frac{\beta}{n}(i-p)} & \text{if } i > p,
    \end{cases}
    \]
and $\mu^3$, where the eigenvalues have been shuffled.

\begin{figure}
  \begin{center}
    \begin{tikzpicture}[scale=0.9]
      \begin{semilogyaxis}[
          width=7.8cm,
          height=7cm,
          ymajorgrids,
          xlabel={Iterations},
          ylabel={residual},
          xmin=0, xmax=60,
          legend cell align=left,
          legend style={at={(0.04,0.03)}, anchor=south west}
        ]
        \addplot [no markers, color=black, line width=1.2pt] table [x expr=\coordindex, y index=0]{data/ipsen__two_distinct__b_evenly_distributed.txt};
        \addlegendentry{$I$-GMRES}
        \addplot+ [no markers, color=brown] table [x expr=\coordindex, y index=1] {data/ipsen__two_distinct__b_evenly_distributed.txt};
        \addlegendentry{$M$-GMRES ($\mu^1$)}
        \addplot+ [no markers, color=blue] table [x expr=\coordindex, y index=2] {data/ipsen__two_distinct__b_evenly_distributed.txt};
        \addlegendentry{$M$-GMRES ($\mu^2$)}
        \addplot+ [no markers, color=red] table [x expr=\coordindex, y index=3] {data/ipsen__two_distinct__b_evenly_distributed.txt};
        \addlegendentry{$M$-GMRES ($\mu^3$)}
      \end{semilogyaxis}
      \end{tikzpicture}
    \caption{Experiment 5}
    \label{fig:experiment5}
  \end{center}
\end{figure}
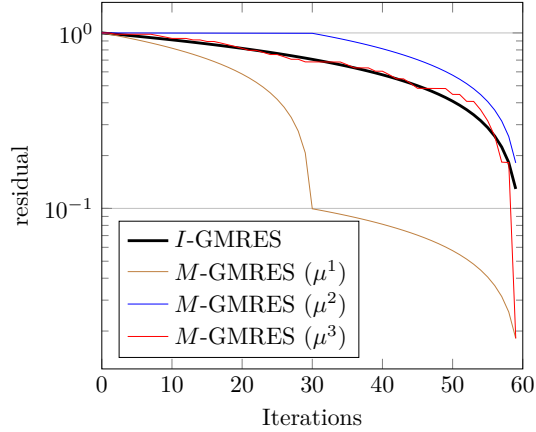

\subsection{Left and right preconditioning}

Note that the results from the previous sections can also be interpreted as instances of left and right preconditioning strategies. In this section, we consider the matrix \texttt{HB/mcfe}\footnote{\url{https://www.cise.ufl.edu/research/sparse/matrices/HB/mcfe.html}} from the University of Florida Sparse matrix collection \cite{zbMATH06721804}. 
It is a real, non-symmetric, sparse matrix of size $765\times 765$ with a condition number \red{(in $\|\cdot\|_I$)} of approximately $10^{14}$.
First, we solve the linear system using unpreconditioned GMRES for two different right-hand sides:
$b=(1, \dots, 1)^\top$ and a random $b$. 
In both cases, the relative residual remains greater than $10^{-8}$ after $n-1=764$ iterations. 
This outcome confirms the challenge posed by this problem for GMRES. 
Moving forward, we will use the random $b$.
Now, we consider two preconditioners: 
(i) ILU: the preconditioner is $H := U^{-1}L^{-1}$, where $L$ and $U$ are the factors in the incomplete LU factorization of $A$, 
(ii) Symmetric part: $H$ is the inverse of $(A+A^*)/2$. 
The corresponding convergence curves are displayed in \Cref{fig:difference_left_right}.
Using ILU as a preconditioner leads to fast convergence in both cases (less than 12 iterations) with an advantage to left preconditioning. On the other hand, using the symmetric part as a preconditioner exhibits a significant difference in convergence speed.
For this case, some metrics given in \Cref{tbl:HA_properties} do not allow us to explain this behaviour.
Note that, from these results, one might be tempted to think that left preconditioning generally offers a faster convergence. However, \Cref{th:inverse_curves} shows this is not the case.
Finally, in \Cref{fig:all_residuals}, we present plots of the preconditioned and unpreconditioned relative residual norms, along with the error.
We use solid lines to display the residual norms actually minimized by left and right preconditioned GMRES: the preconditioned residual $\|Hr\|_I/\|b\|_I$ for left preconditioning and the unpreconditioned residual $\|r\|_I/\|b\|_I$ for right preconditioning. These are the residual to which the stopping criterion respectively applies. The errors produced by these algorithms are plotted in dashed line. These errors are computed from the GMRES approximations at each iteration, compared against the algebraic solution obtained through an LU factorization. Additionally, the residual not minimized by the algorithm is plotted for reference.
In particular, while left preconditioning shows rapid convergence in terms of the preconditioned residual and error, the unpreconditioned residual (\emph{i.e.}\ in Euclidean norm) remains high. \cyan{This is a typical case where monitoring only the left preconditioned residual could lead to early termination if the user is not aware that the Euclidean norm of the residual can still be high.}

\begin{table}[h]
    \centering
	\begin{tabular}{lc}
	\toprule
	$\kappa(HA)$ & $10^7$ \\
	$\kappa(AH)$ & $10^6$ \\
	min/max singular values of $H$ & $10^{-18}$/$10^{-4}$ \\
	  $\|HA\|_I$  & 7 \\
	$\|AH\|_I$ & 7 \\
	\bottomrule
    \end{tabular}
    \caption{Properties of the preconditioned matrix, with \cyan{$H=2(A+A^*)^{-1}$}. \cyan{The condition numbers are defined by $\kappa(X) := \|X\|_I\|X^{-1}\|_I$.}}
\label{tbl:HA_properties}
\end{table}

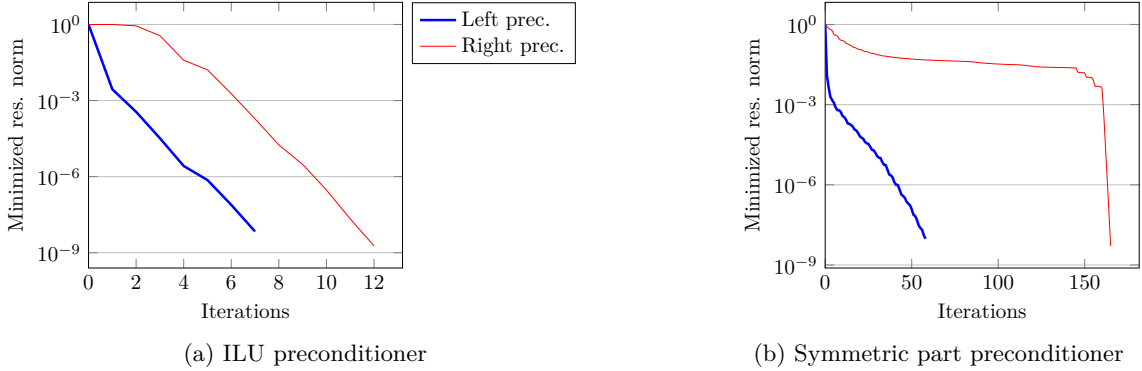
\begin{figure}
  \begin{center}
  \begin{subfigure}[b]{0.49\textwidth} 
    \begin{tikzpicture}[scale=0.8]
      \begin{semilogyaxis}[
          width=6.8cm,
          height=6cm,
          ymajorgrids,
          xlabel={Iterations},
          ylabel={Minimized res. norm},
          xmin=0, 
          legend cell align=left,
          legend pos=outer north east,
        ]
        \addplot [no markers, color=blue, line width=1.2pt] table [x expr=\coordindex, y index=0] {data/difference_left_right_ilu__L.txt};
        \addlegendentry{Left prec.}
        \addplot+ [no markers, color=red] table [x expr=\coordindex, y index=0] {data/difference_left_right_ilu__R.txt};
        \addlegendentry{Right prec.}
      \end{semilogyaxis}
      \end{tikzpicture}
      \caption{ILU preconditioner}
        \label{fig:difference_left_right:ilu}
    \end{subfigure}
    \hfill
    \begin{subfigure}[b]{0.49\textwidth} 
        \centering
        \begin{tikzpicture}[scale=0.8]
      \begin{semilogyaxis}[
          width=6.8cm,
          height=6cm,
          ymajorgrids,
          xlabel={Iterations},
          ylabel={Minimized res. norm},
          xmin=0, 
        ]
        \addplot [no markers, color=blue, line width=1.2pt] table [x expr=\coordindex, y index=0] {data/difference_left_right_sympart__L.txt};
        \addplot+ [no markers, color=red] table [x expr=\coordindex, y index=0] {data/difference_left_right_sympart__R.txt};
      \end{semilogyaxis}
      \end{tikzpicture}
      \caption{Symmetric part preconditioner}
        \label{fig:difference_left_right:sympart}
    \end{subfigure}
    \caption{Preconditioners applied on the left or on the right. The norm that is minimized by the preconditioned GMRES is displayed.}
    \label{fig:difference_left_right}
  \end{center}
\end{figure}

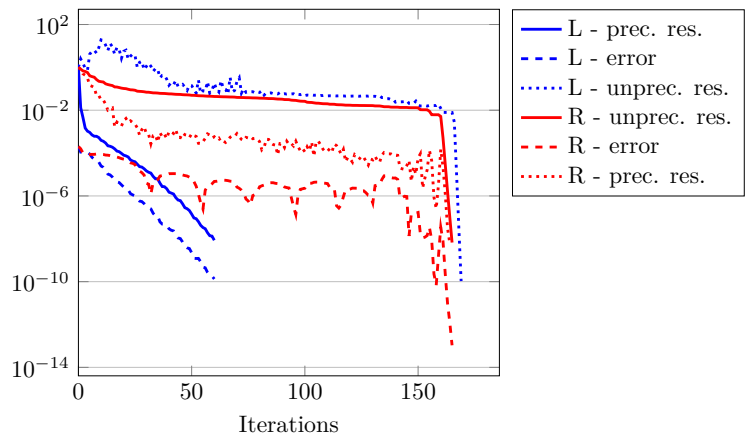
\begin{figure}
  \begin{center}
    \begin{tikzpicture}[scale=0.9]
      \begin{semilogyaxis}[
          width=7.8cm,
          height=7cm,
          ymajorgrids,
          xlabel={Iterations},
          xmin=0, 
          legend cell align=left,
          legend pos=outer north east,
        ]
        \addplot [no markers, color=blue, solid, line width=1.2pt] table [x expr=\coordindex, y index=0]{data/difference_left_right_sympart2_L_min_res.txt};
        \addlegendentry{L - prec. res.}
        \addplot+ [no markers, color=blue, dashed, line width=1.2pt] table [x expr=\coordindex, y index=0] {data/difference_left_right_sympart2_L_error.txt};
        \addlegendentry{L - error}
        \addplot+ [no markers, color=blue, dotted, line width=1.2pt] table [x expr=\coordindex, y index=0] {data/difference_left_right_sympart2_L_nonprec_res.txt};
        \addlegendentry{L - unprec. res.}
        
        \addplot+ [no markers, color=red, solid, line width=1.2pt] table [x expr=\coordindex, y index=0] {data/difference_left_right_sympart2_R_min_res.txt};
        \addlegendentry{R - unprec. res.}
        \addplot+ [no markers, color=red, dashed, line width=1.2pt] table [x expr=\coordindex, y index=0] {data/difference_left_right_sympart2_R_error.txt};
        \addlegendentry{R - error}
        \addplot+ [no markers, color=red, dotted, line width=1.2pt] table [x expr=\coordindex, y index=0] {data/difference_left_right_sympart2_R_prec_res.txt};
        \addlegendentry{R - prec. res.}
      \end{semilogyaxis}
      \end{tikzpicture}
    \caption{Residuals and errors yielded by left (L) and right (R) preconditioned GMRES.}
    \label{fig:all_residuals}
  \end{center}
\end{figure}

\section{Conclusion}

In this work, we extended the theory of prescribed convergence curves for GMRES to the setting of weighted GMRES and preconditioned GMRES. Building on the foundational result of Greenbaum, Pták, and Strakoš, we demonstrated that for any linear system and any prescribed convergence behavior, one can construct a weight matrix or a preconditioning strategy that forces GMRES to exhibit this behavior. 
Our theoretical contributions include a full characterization of weight matrices that yield a desired convergence curve and the necessary and sufficient conditions for realizing two distinct convergence curves simultaneously (one for $I$-GMRES and the other for $M$-GMRES). 
We showed that these results naturally translate to the preconditioned setting, both with split and one-sided preconditioners, highlighting in particular that left and right preconditioning can yield markedly different convergence behaviors, regardless of the spectrum of the preconditioned system. 

\section*{Acknowledgments}

We thank Gérard Meurant for enlightening discussions about left and right preconditioning and for suggesting the HB/mcfe matrix to us. 

\bibliographystyle{plainurl}
\bibliography{references}

\end{document}